\documentclass[final, nomarks]{dmtcs-episciences}
\usepackage{amssymb}
\usepackage{amsthm}
\usepackage{amsmath} 
\usepackage[english]{babel}
\usepackage[T1]{fontenc}    
\usepackage[utf8]{inputenc} 

\usepackage{etoolbox}
\usepackage{enumerate}

\usepackage{wasysym}

\usepackage{url}

\usepackage{subcaption}

\usepackage{silence}
\usepackage{tikz}
\usetikzlibrary{shapes.geometric}

\usepackage{tkz-berge}
\usepackage{tkz-graph}

\theoremstyle{definition}
\newtheorem{maar}{Määritelmä}[section]
\newtheorem{lemma}[maar]{Lemma}

\newtheorem{definition}[maar]{Definition}
\newtheorem{corollary}[maar]{Corollary}
\newtheorem{theorem}[maar]{Theorem}
\newtheorem{example}[maar]{Example}

\newtheorem{fact}{Fact}

\theoremstyle{remark}

\newtheorem*{huom*}{Huomautus}

\newcommand{\abs}[1]{\left\vert{#1}\right\vert}
\newcommand{\sep}{\mid}

\usepackage{mathtools}
\newcommand{\ceil}[1]{\left\lceil #1\right\rceil}

\DeclareMathOperator{\cc}{cc}

\newcommand{\koodit}{C}
\newcommand{\apu}{A}

\newcommand{\koko}{k}
\newcommand{\m}{m}
\newcommand{\K}{m}

\usepackage[mathcal]{euscript}
\newcommand{\clause}{\mathcal{C}}
\newcommand{\instance}{F}
\newcommand{\literals}{U}
\newcommand{\assignment}{A}
\newcommand{\true}{\emph{true }}
\newcommand{\false}{\emph{false }}

\newcommand{\auxiliary}{v_{-1}}
\newcommand{\graphminus}[2]{#1 - #2}
\newcommand{\koodiswap}[3]{#1[#2 \leftarrow #3]}

\usepackage[numbers]{natbib}

\author[Ville Junnila et al.]{Ville Junnila
	\and Tero Laihonen
	\and Havu Miikonen}
\title{New Results on Vertices that Belong to Every Minimum Locating-Dominating Code}
\affiliation{
	Department of Mathematics and Statistics,
	University of Turku, Turku, Finland}
\keywords{locating-dominating code, forced vertex, algorithmic complexity, number of different codes, characterization}
\usepackage{hyperref}

\begin{document}
\publicationdata{vol. 28:2}{2026}{20}{10.46298/dmtcs.16459}{2025-09-03; 2025-09-03; 2026-03-17}{2026-03-19}
\maketitle

\begin{abstract}
\vspace{6pt}
  Locating-dominating codes have been studied widely since their introduction in the 1980s by Slater and Rall. In this paper, we concentrate on vertices that must belong to all minimum locating-dominating codes in a graph. We  call them \emph{min-forced vertices}. We show that the number of min-forced vertices in a connected nontrivial graph of order $n$ is bounded above by  $\frac{2}{3}\left(n -\gamma^{LD}(G)\right)$, where $\gamma^{LD}(G)$ denotes the cardinality of a minimum locating-dominating code. This implies that the maximum ratio between the number of min-forced vertices and the order of a connected nontrivial graph is at most  $\frac{2}{5}$. Moreover, both of these bounds can be attained. In particular, the ratio $\frac{2}{5}$ is obtained by paths of order $5m$ having a unique minimum locating-dominating code of size $2m$. Furthermore, as a natural extension, we determine the number of different minimum locating-dominating codes in paths of all orders. In addition, we show that deciding whether a vertex is min-forced is co-NP-hard. 
\end{abstract}

\section{Introduction}

Let $G = (V(G), E(G))$ be a simple, finite and undirected graph. We often denote the order of a graph $\abs{V(G)}$ by $n$ and the set of vertices $V(G)=\{v_1,v_2,\dots, v_n\}$. A graph is \emph{nontrivial} if $n\ge 2.$ A nonempty subset $S \subseteq V(G)$ is called a \emph{code}.
Elements of a code are called \emph{codewords}. An edge between $u,v \in V(G)$ is denoted by $\{u,v\} = uv$.
We denote by $N(v)$ the \emph{open neighbourhood} of a vertex $v$ which is defined as $N(v) = \{ u\in V \sep vu\in E(G) \}$. The \emph{closed neighbourhood} of a vertex $v \in V$ is $N[v] = N(v) \cup \{v\}$. If we wish to emphasize the underlying graph, we denote $N_G(v)$ and $N_G[v].$
Distinct vertices $u$ and $v$ are called \emph{closed twins} if $N[u] = N[v]$, and \emph{open twins} if $N(u) = N(v)$.
Furthermore, vertices are \emph{twins} if they are either closed or open twins.
By the notation $G - S = G[V(G) \setminus S]$ (respectively, $G - v$ for a single vertex $v$) we mean the graph induced by the vertices of $G$ not including $S$ (resp. $v$).
We use the following shorthand notation: $S[u \leftarrow v] = (S \setminus \{u\}) \cup \{v\}$ for the code where the codeword $u$ is swapped for $v$.
For a subset $E' \subseteq E(G)$, we also use the notation $G[E']$ to denote the graph induced by the edge set $E'$, that is, the graph with the edges of $E'$ and their endpoints.

\begin{definition}
    Let $G$ be a graph and let $S \subseteq V(G)$ be a code.
    The \emph{identifying set}, or $I$-set, of a vertex $v$ is the set of codewords in the closed neighbourhood of $v$. Denote
    \[
        I_G(S;v)  = S \cap N_G[v].
    \]
    We may omit the graph or the code in the notation if they are clear from the context, that is, $I_G(S;v) = I_G(v) = I(S;v) = I(v)$. 
\end{definition}

\begin{definition}
    A code $S \subseteq V(G)$ is a \emph{locating-dominating code}, or \emph{LD-code}, if for all distinct non-codewords $x, y \in V(G) \setminus S$ their $I$-sets are nonempty and
    \(I(x) \neq I(y)\). The cardinality of the minimum locating-dominating codes in a graph $G$ is denoted by $\gamma^{LD}(G)$ and it is called the \emph{locating-dominating number} of $G$.
\end{definition}

Slater and Rall originally introduced locating-dominating codes in the $1980$s, see for example, \cite{rall1984location, slater1987domination, slater1988dominating}.
A lot of research has been done (see the numerous articles in \cite{jeanlobwww})  considering the minimum possible cardinality of a locating-dominating code in graphs. In this paper, we continue this work. In particular, we are interested in certain critical vertices of a graph, which belong to all or no  minimum locating-dominating codes.

\begin{definition}
A \emph{minimum-forced vertex}, or \emph{min-forced vertex}, is a vertex $v \in V(G)$ such that it is in every minimum locating-dominating code in $G$. In other words, $v \in \bigcap_{S \text{ is a minimum LD-code}} S$.  
\end{definition}

Similarly to the min-forced vertices, we define the following related concept.

\begin{definition}
    A \emph{minimum-void vertex}, or \emph{min-void vertex}, is a vertex $v \in V(G)$ such that it is in no minimum locating-dominating code in $G$. In other words, $v \notin \bigcup_{S \text{ is a minimum LD-code}} S$. 
\end{definition}

These  vertices have been considered in the case of \emph{trees}, for example, in Blidia and Lounes~\cite{Blidia2009minforcedtrees}. Similar concepts for more general graphs have been considered in other contexts as well, see, for example,  \cite{Boros} for stable sets and  \cite{Hakanen_2022} for metric bases. Notice also the resemblance of the min-forced vertices to the core vertices of a dominating set in \cite{BOUQUET20219}.

In this paper, we first show a characterization of min-forced vertices that reduces the problem to LD-codes in smaller subgraphs. In Section~\ref{ratiosection}, we provide, for example, the maximum ratio between the number of min-forced vertices and the order of a graph. We show that the maximum ratio is $\frac{2}{5}$.  The ratio $\frac{2}{5}$ is attained by paths of order $5k$ having a unique locating-dominating code of size $2k$.  In Section~\ref{pathsection}, we further determine the exact number of different minimum locating-dominating codes in paths of all orders. In the final section, we also show that it is computationally challenging to decide whether or not a vertex is min-forced in a graph. Notice that algorithmic complexity in the case where all the codewords are min-forced, \textit{i.e.}, the LD-code is unique, is discussed in \cite{Hudry2019}.

\section{A characterization of min-forced vertices}

In the following theorem, we give a characterization of a min-forced vertex.

\begin{theorem}
\label{theo:forced-characterization}
    Let $G$ be a graph.
    A vertex $v \in V(G)$ is min-forced if and only if either
    \begin{enumerate}[(i)]
        \item $v$ is isolated, 
        \item $\gamma^{LD}(G-v) > \gamma^{LD}(G)$, or
        \item  $\gamma^{LD}(G-v) = \gamma^{LD}(G)$, and for the graph $G-v$ there exists no minimum LD-code $S$ such that $I_G(S;v) \neq \emptyset$ and $I_G(S;v)\neq I_G(S;w)$ for all $w \in V(G) \setminus (S \cup \{v\})$.
    \end{enumerate}
\end{theorem}
\begin{proof}
    First, we assume that one of the conditions (i)--(iii) holds and show that then $v$ is min-forced.
    \begin{enumerate}[(i)]
        \item If $v$ is isolated, then it is min-forced.
        
        \item If the graph $G-v$ cannot be locating-dominated using $\gamma^{LD}(G)$ vertices, then clearly the graph $G$ cannot be locating-dominated using $\gamma^{LD}(G)$ vertices without using the vertex $v$, therefore $v$ is min-forced in $G$.

        \item 
        Assume that for all minimum LD-codes $S$ in the graph $G-v$ we have either $I_G(S;v) = \emptyset$ or $I_G(S;v) = I_G(S;w)$ for some $w \in V(G) \setminus ( S \cup \{v\})$.
        If there is an LD-code $S'$ in $G$ of order $\gamma^{LD}(G)$ such that $v \notin S'$, then $S'$ is also a minimum LD-code in $G-v$.
        This is a contradiction: we assumed that no minimum LD-code in $G-v$ both distinguishes and dominates $v$.
        Therefore, the graph $G$ cannot have minimum LD-codes $S'$ that do not contain $v$, in other words, $v$ is min-forced in $G$.
    \end{enumerate}

    Next, we assume that $v$ is min-forced and we show that one of the conditions (i)--(iii) follows.
    When the vertex $v$ is removed from $G$, the locating-dominating number of the graph may increase, decrease or remain the same.
    We consider each case separately.

    \begin{enumerate}[(a)]
        \item Suppose first that $\gamma^{LD}(G-v) < \gamma^{LD}(G)$.
        We will show that the vertex $v$ is isolated, meeting the condition (i).
        
        Let $S$ be an LD-code in $G-v$ such that $\abs{S} = \gamma^{LD}(G)-1$.
        Notice that such a code always exists.
        If the vertex $v$ has no neighbours in $S$ (in other words, $I_G(S;v) = \emptyset$), then there are two possible cases:
        \begin{enumerate}[(1)]
            \item If $N_G(v) = \emptyset$, then $v$ is isolated.
            \item There exists a vertex $u \in N_G(v)$.
            We know that $u \notin S$, because the $I$-set of $v$ under $S$ is empty.
            The set $S \cup \{u\}$ is an LD-code in $G$, because the $I$-sets of all  vertices not contained in $S \cup \{u\}$ are distinct and nonempty:
            no vertex $w \in V(G) \setminus \{v\}$ has $I(S \cup \{u\}; w) = \{u\}$, because then $I(S; w) = \emptyset$, a contradiction.
            The $I$-sets that are distinct under the code $S$ are distinct under $S \cup \{u\}$. 
            This contradicts the fact that $v$ is min-forced.
        \end{enumerate}
        If the vertex $v$ has neighbours in $S$, then $I_G(S; v) \neq \emptyset$. Furthermore, $I_G(S; v) = I_G(S; w)$ for at least one vertex $w \in V(G) \setminus (S \cup \{v\})$, since otherwise $S$ would be an LD-code in $G$, leading to a contradiction as $\abs{S} < \gamma^{LD}(G)$.  If there were two or more such vertices with the same $I$-set, then $S$ would not be locating-dominating in $G-v$ (a contradiction). Hence, we may assume that there is exactly one such vertex $w$. 
        Now, $S \cup \{w\}$ is a minimum LD-code in $G$.
        Therefore, $v$ is not min-forced, which is a contradiction.

        \item If $\gamma^{LD}(G-v) > \gamma^{LD}(G)$, the condition (ii) holds.

        \item Assume finally that $\gamma^{LD}(G-v) = \gamma^{LD}(G)$.
        
        If there exists a minimum LD-code $S$ for $G-v$ such that $I_G(S; v) \neq \emptyset$ and $I_G(S; v)\neq I_G(S; w)$ for all $w \in V(G) \setminus (S \cup \{v\})$, then $S$ is also an LD-code in $G$.
        The vertex $v$ does not belong to all minimum LD-codes in $G$, in particular, $v$ does not belong to $S$.
        This contradicts the assumption that $v$ is min-forced.
    \end{enumerate}\vspace{-2\baselineskip}
\end{proof}

In the next example, we demonstrate how Theorem~\ref{theo:forced-characterization} can be used to show that a vertex is min-forced.

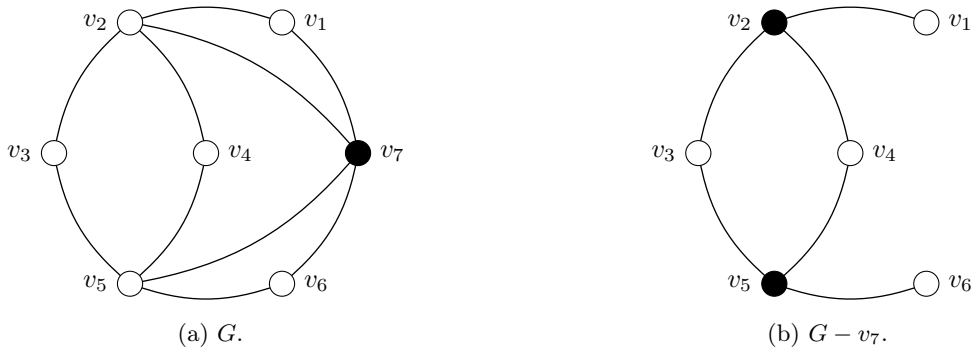
\begin{figure}
    \centering
     \begin{subfigure}[t]{0.49\textwidth}
        \centering
        \begin{tikzpicture}[yscale=2, xscale=2, rotate=0]

    \renewcommand*{\EdgeLineWidth}{ 0.5pt}
    \SetVertexMath
    \SetVertexLabelOut

    {
    \SetVertexNoLabel
    \tikzstyle{VertexStyle}=[minimum size=1pt,inner sep=0pt]
    \Vertices{circle}{1,2,3,4,5,6}
    \Vertex[x=0, y=0]{7}
    }

    \tikzstyle{EdgeStyle}=[bend right=20]
    \Edge(1)(2)
    \Edge(1)(3)
    \Edge(5)(1)
    \Edge(6)(1)
    \Edge(2)(3)
    \Edge(3)(4)
    \Edge(4)(5)
    \Edge(5)(6)
    \Edge(7)(3)
    \Edge(5)(7)

    \tikzstyle{VertexStyle}=[circle,draw, fill=black]
    \SO[unit=0](1){v_7}

    \tikzstyle{VertexStyle}=[circle,draw, fill=white]
    \SO[unit=0](2){v_1}
    \SO[unit=0, Lpos=180](3){v_2}
    \SO[unit=0, Lpos=180](5){v_5}
    \SO[unit=0](6){v_6}
    \SO[unit=0, Lpos=180](4){v_3}
    \SO[unit=0](7){v_4}
    
\end{tikzpicture}
        \caption{$G$.}
     \end{subfigure}
     \hfill
     \begin{subfigure}[t]{0.49\textwidth}
        \centering
        \begin{tikzpicture}[yscale=2, xscale=2, rotate=0]

    \renewcommand*{\EdgeLineWidth}{ 0.5pt}
    \SetVertexMath
    \SetVertexLabelOut

    {
    \SetVertexNoLabel
    \tikzstyle{VertexStyle}=[minimum size=1pt,inner sep=0pt]
    \Vertices{circle}{1,2,3,4,5,6}
    \Vertex[x=0, y=0]{7}
    }

    \tikzstyle{EdgeStyle}=[bend right=20]

    \Edge(2)(3)
    \Edge(3)(4)
    \Edge(4)(5)
    \Edge(5)(6)
    \Edge(7)(3)
    \Edge(5)(7)

    \tikzstyle{VertexStyle}=[circle,draw, fill=black]
    \SO[unit=0, Lpos=180](3){v_2}
    \SO[unit=0, Lpos=180](5){v_5}

    \tikzstyle{VertexStyle}=[circle,draw, fill=white]
    \SO[unit=0](2){v_1}
    \SO[unit=0](6){v_6}
    \SO[unit=0, Lpos=180](4){v_3}
    \SO[unit=0](7){v_4}

\end{tikzpicture}
        \caption{$\graphminus{G}{v_7}$.}
		
     \end{subfigure}
    \caption{The black vertices are shown to be min-forced in their respective graphs.}

    \label{fig:deletion-example}
\end{figure}

\begin{example}
    Consider the graph $G$ illustrated in Figure~\ref{fig:deletion-example}(a).
    The location-domination number of $G$ is $3$; the set $\{v_2, v_3, v_7\}$ is an LD-code in $G$, and two codewords are not enough to distinguish the remaining 5 vertices, since only three nonempty $I$-sets can be formed using two codewords.
    We claim that the vertex $v_7$ is min-forced in $G$.
    Let us examine the graph $\graphminus{G}{v_7}$, illustrated in Figure~\ref{fig:deletion-example}(b), and its minimum LD-codes. 

    The vertex $v_2$ is min-forced in $\graphminus{G}{v_7}$ by Theorem~\ref{theo:forced-characterization}(ii).
    Indeed, the graph $\graphminus{G}{\{v_7, v_2\}}$ is composed of an isolated vertex and the star graph $K_{1,3}$, thus $\gamma^{LD}(\graphminus{G}{\{v_7, v_2\}}) = 1 + 3 > 3$.
    By the same argument, the vertex $v_5$ is min-forced in $\graphminus{G}{v_7}$ as well.
    Therefore, all minimum LD-codes $S$ in $\graphminus{G}{v_7}$ must contain both $v_2$ and $v_5$.

    The vertices $v_3$ and $v_4$ are twins, so to distinguish them, at least one of them must be in each LD-code in $\graphminus{G}{v_7}$.
    Together with the two min-forced vertices, this gives us $\gamma^{LD}(\graphminus{G}{v_7}) \geq 3$.
    It is easy to verify that the sets $\{v_2, v_3, v_5\}$ and $\{v_2, v_4, v_5\}$ are LD-codes of order $3$ in $\graphminus{G}{v_7}$, and by the reasoning above, they are the only minimum LD-codes in $\graphminus{G}{v_7}$.
    Additionally, we see that $\gamma^{LD}(G-v_7) = \gamma^{LD}(G)$, and to show that $v_7$ is min-forced in $G$, we use Theorem~\ref{theo:forced-characterization}(iii).

    Now we consider the minimum LD-codes of $\graphminus{G}{v_7}$ in the graph $G$.
    For the code $S_1 = \{v_2, v_3, v_5\}$, we observe that
    $I_G(S_1; v_7) = \{v_2, v_5\} = I_G(S_1; v_4)$.
    Similarly, for $S_2 = \{v_2, v_4, v_5\}$, we see that $I_G(S_2; v_7) = \{v_2, v_5\} = I_G(S_2; v_3)$.
    We conclude that there are no minimum LD-codes in $\graphminus{G}{v_7}$ that distinguish the vertex $v_7$ in $G$; the condition (iii) of Theorem~\ref{theo:forced-characterization} is met and therefore, $v_7$ is min-forced in $G$.
\end{example}

As we saw above, determining whether a vertex $v \in V(G)$ is min-forced, can involve the computation of the location-domination number of a graph of order $\abs{V(G)}-1$, which can be tedious in general.
In Section~\ref{sec:complexity}, we prove that determining whether a vertex is min-forced is actually a co-NP-hard problem. Therefore, a simple method regarding a min-forced vertex \emph{in general} can be hard to find. However, for \emph{trees}, a characterization of min-forced vertices and an efficient algorithm for determining them is given in Lemma~3 and Theorem~3 of \cite{Blidia2009minforcedtrees}.

\section{A tight bound for the maximum number of min-forced vertices}
\label{ratiosection}

A natural question regarding the min-forced and min-void vertices is the following. 
What is the maximum ratio of such vertices in a graph with respect to the order $n$? 
For min-void vertices this is easily answered.
Indeed, for any $h\ge 2$, take the graph $G_h$ of order $n=2^h-1+h$,  where $h$ vertices form an independent set and the rest of the vertices have the $2^h-1$ different and nonempty subsets of the independent set as their neighbourhoods.
It is well known (and also easy to check) that the independent set is a minimum LD-code giving $\gamma^{LD}(G_h)=h$.
By \cite[Lemma 7]{Hudry2019}, the LD-code is unique, and hence, the graph has as many min-void vertices as there are non-codewords (that is, the maximum amount $n-\gamma^{LD}(G_h)=2^h-1$). Notice that here $\gamma^{LD}(G_h)$ is as small as possible with respect to $n$ since for any graph it holds that $n\le 2^{|S|}-1+|S|$ if $S$ is a locating-dominating code.
Thus, the maximum number $k$ of min-void vertices satisfies
$$ k=\frac{2^h-1}{2^h-1+h}n,$$
where the ratio $k/n$ approaches $1$ as $h$ grows.

However, the question for min-forced vertices is a harder problem. This question will be answered in Corollary~\ref{Cor_max_number_of_min-forced}.
To that end, inspired by \cite{hernando2018locating} (and to some extent also by~\cite{Hakanen_2022} for resolving sets), we introduce the concept of the colour graph associated with a locating-dominating code.
Our colour graph is different in essential ways from the the so-called $S$-associated graph defined in \cite{hernando2018locating} (and the colour graph of \cite{Hakanen_2022}), for example, it contains the codewords (as vertices) as well as their edges between non-codewords and also has a crucial auxiliary vertex, which is not a vertex in the original graph. The extension of the $S$-associated graph of \cite{hernando2018locating} to our colour graph is needed in order to prove the optimal result in Theorem~\ref{theo:forced-40-bound}.

\begin{definition} 
    Let $G$ be a graph and let $S \subseteq V(G)$ be a locating-dominating code in $G$.
    The vertices of $G$ are denoted $V(G) = \{v_1, \dots, v_n\}$.
    We define the \emph{colour graph} $G_S$ as follows: $V(G_S) = V(G) \cup \{\auxiliary\}$ and \[E(G_S) = \left\{\{x, y\} \in \binom{V(G_S)}{2} \sep u \in S,\ x, y \notin (S \setminus \{u\}) \text{ and } I_G(S\setminus \{u\}; x) = I_G(S\setminus \{u\}; y)\right\},\]
    where we interpret that $I_G(S'; \auxiliary) = \emptyset$ for all codes $S'\subseteq V(G)$.
    We call vertices in $V(G)$ \emph{true vertices} and the vertex $\auxiliary$ an \emph{auxiliary vertex}.
    Associating a colour with each codeword $u \in S$, we assign the colour $u$ to the edge $xy \in E(G_S)$ if $I_G(S\setminus \{u\}; x) = I_G(S\setminus \{u\}; y)$.
\end{definition}
An example of a colour graph is given in Figure~\ref{fig:värigraafiesim}.
Notice that $G_S$ is a simple graph, not a multigraph, and that each edge has exactly one colour.
The idea behind the colour graph is that if there is an edge $xy \in E(G_S)$ with colour $u$, then $u$ is the only codeword that distinguishes $x$ and $y$.
If either $x$ or $y$ is the auxiliary vertex $\auxiliary$, say, $y = \auxiliary$, then $u$ is the only codeword that dominates $x$.
For edges $ux \in E(G_S)$, where $ u \in S$, the interpretation is that if the vertex $u$ was removed from $S$, then the $I$-sets  of $u$ and $x$ under the code $S \setminus \{u\}$ would be equal.
Finally, the edge $u\auxiliary \in E(G_S)$ means that $u$ has no codeword neighbours in $G$.
These observations immediately imply the following result.

\begin{lemma}
    \label{lemma:at-least-1-edge}
    Let $G$ be a graph and let $S$ be a minimal locating-dominating code in $G$.
    Then for each codeword $u$ there is an edge in $G_S$ with the colour $u$.
\end{lemma}

\begin{figure}
    \centering
     \begin{subfigure}[t]{0.49\textwidth}
        \centering
        \begin{tikzpicture}[yscale=0.55, xscale=0.5, rotate=0]

    \renewcommand*{\EdgeLineWidth}{ 0.5pt}
    \SetVertexMath
    \SetVertexLabelOut

    {
    \SetVertexNoLabel
    \tikzstyle{VertexStyle}=[minimum size=1pt,inner sep=0pt]

    \Vertex[x=8, y=0]{1}
    \Vertex[x=4, y=2.67]{2}
    \Vertex[x=8, y=2.67]{3}
    \Vertex[x=4, y=8]{4}
    \Vertex[x=12, y=2.67]{5}
    \Vertex[x=0, y=5.33]{6}
    \Vertex[x=0, y=2.67]{7}
    \Vertex[x=4, y=0]{8}
    \Vertex[x=4, y=5.33]{9}

    }

    \tikzstyle{EdgeStyle}=[bend right=0]
    \Edge(1)(5)
    \Edge(1)(8)
    \Edge(2)(6)
    \Edge(2)(7)
    \Edge(2)(8)
    \Edge(3)(8)
    \Edge(4)(6)
    \Edge(4)(7)    
    \Edge(4)(9)
    \Edge(6)(7)
    \Edge(6)(8)
    \Edge(6)(9)
    \Edge(7)(8)
    \Edge(7)(9)
    \tikzstyle{EdgeStyle}=[bend right=40, looseness=1.1]
    \Edge(8)(9)

    \tikzstyle{VertexStyle}=[circle,draw, fill=black]
    \SO[unit=0, Lpos=90](1){v_1}
    \SO[unit=0, Lpos=90](2){v_2}
    \SO[unit=0, Lpos=180](7){v_7}
    \SO[unit=0, Lpos=180](8){v_8}

    \tikzstyle{VertexStyle}=[circle,draw, fill=white]
    \SO[unit=0, Lpos=90](3){v_3}
    \SO[unit=0](4){v_4}
    \SO[unit=0, Lpos=90](5){v_5}
    \SO[unit=0, Lpos=180](6){v_6}
    \SO[unit=0](9){v_9}

\end{tikzpicture}
        \caption{The graph $G$ with a minimum LD-code \mbox{$S = \{v_1, v_2, v_7, v_8\}$}.}
     \end{subfigure}
     \hfill
     \begin{subfigure}[t]{0.49\textwidth}
        \centering
        \begin{tikzpicture}[yscale=0.55, xscale=0.5, rotate=0]

    \renewcommand*{\EdgeLineWidth}{ 0.5pt}
    \SetVertexMath
    \SetVertexLabelOut

    {
    \SetVertexNoLabel
    \tikzstyle{VertexStyle}=[minimum size=1pt,inner sep=0pt]

    \Vertex[x=8, y=0]{1}
    \Vertex[x=4, y=2.67]{2}
    \Vertex[x=8, y=2.67]{3}
    \Vertex[x=4, y=8]{4}
    \Vertex[x=12, y=2.67]{5}
    \Vertex[x=0, y=5.33]{6}
    \Vertex[x=0, y=2.67]{7}
    \Vertex[x=4, y=0]{8}
    \Vertex[x=4, y=5.33]{9}

    }

    \tikzstyle{VertexStyle}=[circle,draw, fill=white]
    \Vertex[x=8, y=6.33, Lpos=45]{\auxiliary}

    \tikzstyle{LabelStyle}=[fill=white, circle,minimum size=12pt,inner sep=1pt,scale=.8]

    \tikzstyle{EdgeStyle}=[black]
    \Edge[label=$v_1$](1)(3)
    \Edge[label=$v_1$](5)(\auxiliary)
    \Edge[label=$v_2$](2)(6)
    \Edge[label=$v_2$](2)(9)
    \Edge[label=$v_2$](6)(9)
    \Edge[label=$v_7$](\auxiliary)(4)
    \Edge[label=$v_7$](3)(9)
    \Edge[label=$v_7$](7)(6)
    \Edge[label=$v_8$](3)(\auxiliary)
    \Edge[label=$v_8$](4)(9)

    \tikzstyle{VertexStyle}=[circle,draw, fill=black]
    \SO[unit=0](1){v_1}
    \SO[unit=0, Lpos=0](2){v_2}
    \SO[unit=0, Lpos=180](7){v_7}
    \SO[unit=0, Lpos=180](8){v_8}

    \tikzstyle{VertexStyle}=[circle,draw, fill=white]
    \SO[unit=0](3){v_3}
    \SO[unit=0, Lpos=180](4){v_4}
    \SO[unit=0, Lpos=180](5){v_5}
    \SO[unit=0, Lpos=180](6){v_6}
    \SO[unit=0](9){v_9}

\end{tikzpicture}
        \caption{The colour graph $G_S$.}
		
     \end{subfigure}
    \caption{The vertex $v_8$ is min-forced as can be verified using Theorem~\ref{theo:forced-characterization}(ii).}

    \label{fig:värigraafiesim}
\end{figure}
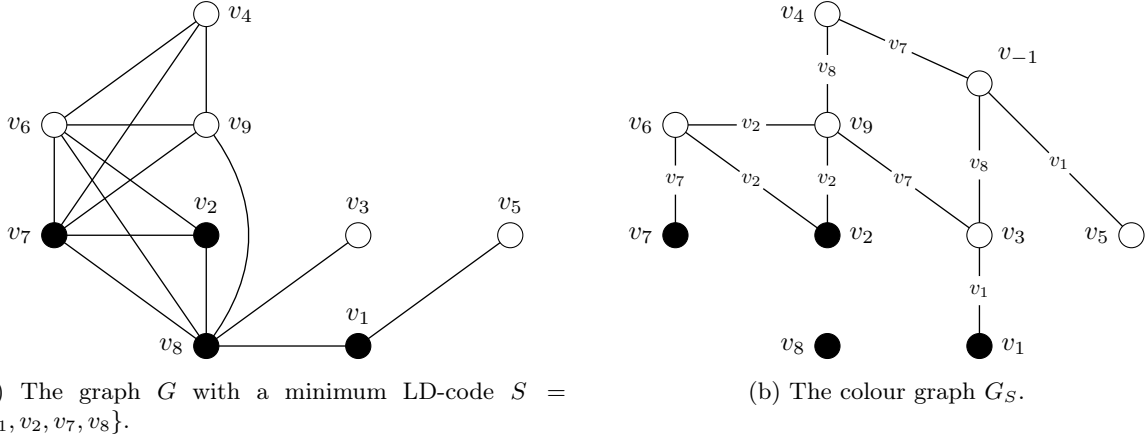

Despite the differences between the colour graph given here and the graphs considered in~\cite{hernando2018locating} and \cite{Hakanen_2022}, they share many similar properties like the ones stated in Proposition 5 of~\cite{hernando2018locating} and Lemma 25 of~\cite{Hakanen_2022}, as we shall see in the next lemma.

\begin{lemma}
\label{lemma:G_S-properties}
Let $G$ be a graph and let $S$ be a locating-dominating code in $G$.
    \begin{enumerate}[(i)]
        \item There are no edges $uv \in E(G_S)$ such that $u, v \in S$.
        
        \item Edges $uv$ of $G_S$
        with $u \in S$
        have the colour $u$.
        
        \item Two edges of $G_S$ between non-codewords that share an endpoint have distinct colours.
        
        \item If there is an edge $xy \in E(G_S)$ with the colour $u$, where $x, y \notin S$, then exactly one of $x$ and $y$ is a neighbour of $u$ in $G$.
        This holds regardless of whether $x$ and $y$ are both true vertices or not.
        
        \item Let $u$ be a codeword and $x, y \in V(G_S) \setminus S$.
        If there exists two of the three edges $ux$, $uy$ and  $xy$  in $E(G_S)$ with the colour $u$, then the remaining edge exists and has the colour $u$.

        \item There can be at most two edges incident with $u \in S$ in $G_S$.
        
        \item Every cycle in $\graphminus{G_S}{S}$ has an even number of edges with the colour $u$ for each $u$ appearing on the cycle. The graph $\graphminus{G_S}{S}$ is bipartite.
        
        \item Let $W = w_1 w_2 \cdots w_k$ be a walk with no repeated edges in $\graphminus{G_S}{S}$. If $W$ has an even number of edges with the colour $u$ for each $u$ appearing in the walk, then it is a closed walk.
    \end{enumerate} 
\end{lemma}
\begin{proof}
    (i) This claim trivially follows from the definition.
    
    (ii) This follows from the definition, because for any other codeword $w \in S$, the condition $u, v \notin (S \setminus \{w\})$ does not hold.
    
    (iii)--(iv) Assume that there is an edge $xy \in E(G_S)$ with the colour $u$, where $x, y \notin S$.
    Then the $I$-sets of $x$ and $y$ differ only by $u$.
    Without loss of generality, denote $I(x) = X \cup \{u\}$ and $I(y) = X$, where $u \notin X$ and $X \subseteq S$.
    By the definition of $I$-sets, $x \in N_G(u)$ and $y \notin N_G(u)$ if $y$ is a true vertex.
    If $y = \auxiliary$, then $X = \emptyset$.
    In that case $x \in N_G(u)$.
    This proves the claim~(iv).
    
    Suppose there is  a vertex $z \notin S$ with a $u$-coloured edge $zx$ in $G_S$. Its $I$-set differs from that of $x$ only by the vertex $u$, that is, $I(z) = X = I(y)$, a contradiction with $S$ being a locating-dominating code, as $y$ and $z$ are not codewords.
    If $y = \auxiliary$, then $I(z) = \emptyset$, also a contradiction.
    
    Similarly, if there is an edge $zy \in E(G_S)$ with the colour $u$, then $I(z) = I(y) \cup \{u\} = I(x)$, a contradiction.
    This proves the claim~(iii).

    (v) Any two of the $u$-coloured edges $ux$, $uy$ and  $xy$ in $E(G_S)$ imply two of the following equalities:
    \begin{align*}
        I(S\setminus\{u\}; u)&=I(S\setminus\{u\}; x) \\
        I(S\setminus\{u\}; u)&=I(S\setminus\{u\}; y) \\
        I(S\setminus\{u\}; x)&=I(S\setminus\{u\}; y),
    \end{align*}
    where $I(S \setminus \{u\};x) = \emptyset$ or $I(S \setminus \{u\};y) = \emptyset$ if $x = v_{-1}$ or $y = v_{-1}$, respectively.
    The third equality follows from transitivity and in turn, implies the existence of the corresponding edge in $G_S$ with the colour $u$.

    (vi) Suppose there are three edges $ux$, $uy$ and $uz$ in $G_S$.
    Due to~(i), we can assume that $x, y, z \notin S$. 
    Moreover, by~(ii), each edge has the colour $u$.
    By~(v), there are edges $xy$ and $xz$ in $E(G_S)$ with the colour $u$, contradicting~(iii).

    (vii) 
    Let $C =  w_1 w_2 \cdots w_k w_1$ be a cycle in $\graphminus{G_S}{S}$.
    Each edge $w_{i} w_{i+1}$ in $C$ corresponds to the removal (or addition) of a codeword $u$ from (or to) $I(w_i)$ to get $I(w_{i+1})$. 
    Each such removal (addition) has to be undone by adding (removing) the vertex $u$ somewhere along the cycle to get back to $I(w_i)$.
    Thus, each colour $u$ appears on the cycle $C$ an even number of times, and by the well-known result by K\H{o}nig,  $\graphminus{G_S}{S}$ is bipartite.
    
    (viii) Each appearance of a colour $u$ along the walk $W =  w_1 w_2 \cdots w_k$ adds or removes the codeword $u$ from the $I$-set of a vertex on the walk.
    Starting from $I(w_1)$, the additions and removals of codewords cancel each other out and we get $I(w_k) = I(w_1)$.
    If $I(w_1) = I(w_k)$, then $w_1 = w_k$ by $S$ being a locating-dominating code.
    Thus, $W$ is a closed walk.
\end{proof}

In the following definition, we introduce a weaker version of min-forced vertices in order to show a stronger result in Lemma~\ref{lemma:atleast2edges}.

\begin{definition}
    Let $G$ be a graph and let $S$ be a locating-dominating code in $G$. 
    The vertex $v \in S$ is \emph{non-swappable} with respect to $S$, if $\koodiswap{S}{v}{u}$ is not a locating-dominating code for any $u \neq v$.
\end{definition}

\begin{lemma}
\label{lemma:atleast2edges}
Let $G$ be a connected nontrivial graph and let $S$ be a minimal locating-dominating code in $G$.
If $v \in S$ is non-swappable  with respect to $S$, then there are at least 2 edges with the colour $v$ in the graph  $\graphminus{G_S}{S}$. 
\end{lemma}

\begin{proof}
We shall show that a codeword $u$ with fewer than two edges of colour $u$ in $\graphminus{G_S}{S}$ can be swapped for another vertex while maintaining the locating-dominating properties of the code.
We analyse cases based on how many $u$-coloured edges there are in the entire colour graph $G_S$.
        \begin{enumerate}[(1)]\addtocounter{enumi}{-1}
        \item If there are no edges in $G_S$ with the colour $u$, then $S \setminus \{u\}$ is an LD-code in $G$, contradicting the minimality of $S$.
        
        \item Suppose then that there exists exactly one edge with the colour $u$ in $G_S$.
        If the only edge with the colour $u$ in $G_S$ is of type $xy$, where $x, y \notin S$ and $x \in N_G(u)$ (by Lemma~\ref{lemma:G_S-properties}(iv)), then $\koodiswap{S}{u}{x}$ is an LD-code.
        Indeed, the only pair of vertices that can violate the location-domination property in $S \setminus \{u\}$ is the pair $x$ and $y$.
        Therefore, the code  $\koodiswap{S}{u}{x}$  is locating-dominating since $x$ is now a codeword.
        This holds regardless of whether $y$ is the auxiliary vertex or a true vertex in $G$.
        It follows that $u$ is \emph{not} non-swappable.

        If the only edge is of type $ux$, where $x \notin S$ and $x \neq \auxiliary$, then the set $\koodiswap{S}{u}{x}$ is an LD-code.
        If $x = \auxiliary$, then let $w$ be any vertex in $N_G(u)$. (Indeed, such a vertex $w$ exists because there are no isolated vertices in $G$, and none of the neighbours of $u$ are in $S$ because $\auxiliary$ is a neighbour of $u$ in $G_S$.)
        The set $\koodiswap{S}{u}{w}$ is an LD-code and thus, the vertex $u$ is not non-swappable.

        \item

        Let there be exactly two edges with the colour $u$ in $G_S$.
        If all four endpoints of the two edges are in $\graphminus{G_S}{S}$, we are done; the vertex $u$ may or may not be non-swappable. 
        Now, let at least one of the edges be incident with a vertex in $S$, and by Lemma~\ref{lemma:G_S-properties}(ii), the codeword in $S$ must be $u$.

        If there are two edges $ux$ and $uy$ (or  $ux$ and $xy$) with the colour $u$, then by Lemma~\ref{lemma:G_S-properties}(v) there is a third edge $xy$ (respectively, $uy$) with the colour $u$, contradicting the assumption that there are exactly two such edges.
        
        Next, we consider the case where the edges are of type $xy$ and $uz$, where $x, y, z \notin S$ and $z \notin \{x, y\}$.
        Let one of $x$ or $y$ be the auxiliary vertex, say, $y = \auxiliary$.
        Now $I(S; x) = \{u\}$, and therefore $x \in N_G(u)$.
        If $x \notin N_G(z)$, then the code $\koodiswap{S}{u}{x}$ is still an LD-code, because $z$ is dominated by codewords other than $u$ (there is no edge $z\auxiliary = zy$ with the colour $u$ in $G_S$) and $x$ distinguishes $z$ and $u$.
        On the other hand, if $x \in N_G(z)$, then $\koodiswap{S}{u}{z}$ is an LD-code.
        Now $I(\koodiswap{S}{u}{z};x) = \{z\}$, which is a unique $I$-set with respect to $\koodiswap{S}{u}{z}$.
        The vertex $u$ is dominated by some codeword (there is no edge $u\auxiliary$ in $G_S$), and finally, $u$ is distinguished (notice that $z$ is now a codeword).

        Now we may assume that $x$ and $y$ are true vertices.
        By Lemma~\ref{lemma:G_S-properties}(iv), we may assume without loss of generality that $x \in N_G(u)$ and $y \notin N_G(u)$. 
        Let us first assume that $z \neq \auxiliary$. Since there are exactly two edges with the colour $u$ in $G_S$, we know  that all the pairs of vertices except $\{x,y\}$ and $\{u,z\}$ are distinguished and all vertices are dominated by codewords in $S\setminus \{u\}$. 
        There are three different cases depending on whether the vertex $z$ is adjacent to $x$ or $y$ in $G$:
        \begin{enumerate}
            \item If $\abs{N_G(z) \cap \{x, y\}} = 1$, then $\koodiswap{S}{u}{z}$  is an LD-code. Indeed, now the codeword $z$ distinguishes the pair $\{x,y\}$, and the pair $\{u,z\}$ does not need to be considered since $z$ is a codeword in the LD-code $\koodiswap{S}{u}{z}$. 
            \item If $N_G(z) \cap \{x, y\} = \{x, y\}$, then $\koodiswap{S}{u}{y}$ is an LD-code. Now the codeword $y$ distinguishes the pair $\{u,z\}$ and the pair $\{x,y\}$ does not need to be considered as $y\in \koodiswap{S}{u}{y}$.
            \item If $N_G(z) \cap \{x, y\} = \emptyset$, then $\koodiswap{S}{u}{x}$ is an LD-code. Now the codeword $x$ distinguishes the pair $\{u,z\}$ and the pair $\{x,y\}$ does not need to be considered as $x\in \koodiswap{S}{u}{x}$.
        \end{enumerate}
        If $z = \auxiliary$, then $\koodiswap{S}{u}{x}$ is an LD-code.

        \item Assume finally that there are at least three edges of colour $u$.
        By Lemma~\ref{lemma:G_S-properties}(ii), if a $u$-coloured edge is incident with a codeword, the codeword must be $u$.
        According to Lemma~\ref{lemma:G_S-properties}(vi), there cannot be three edges $ux$, $uy$ and $uz$.
        Thus, there must be at least one $u$-coloured edge in $\graphminus{G_S}{S}$.
        If two $u$-coloured edges have both their endpoints in $\graphminus{G_S}{S}$, we are done.
        Hence, we may assume that there is exactly one such edge.
        There are at least two other edges with the colour $u$, so they must be $ux$ and $uy$ for some $x, y \notin S$.
        By Lemma~\ref{lemma:G_S-properties}(v), the edge $xy$ is in $E(G_S)$ and it has the colour $u$.
        Therefore, the only $u$-coloured edge in $\graphminus{G_S}{S}$ is the edge $xy$.
        Lemma~\ref{lemma:G_S-properties}(iv) lets us assume that $x \in N_G(u)$.
        If $xy \in E(G)$, then $\koodiswap{S}{u}{y}$  is an LD-code.
        Otherwise $\koodiswap{S}{u}{x}$ is an LD-code, regardless of whether $y$ is a true vertex or not.
        In any case, the vertex $u$ is not non-swappable.

        Any other choice of three or more edges with the colour $u$ results in at least two edges with the colour $u$ between non-codewords in $G_S$.

    \end{enumerate}
    In conclusion, codewords with fewer than 2 edges in $\graphminus{G_S}{S}$ can be swapped for another vertex to get an LD-code, therefore, non-swappable vertices in a code $S$ have at least 2 edges in their colour in $\graphminus{G_S}{S}$.
\end{proof}

The converse of the claim in the previous lemma does not hold in general as will be seen next. There are graphs $G$ with locating-dominating codes $S$ such that there are ``swappable vertices'' that have at least two colour edges between non-codewords in $G_S$.
For example, the colour graph in Figure~\ref{fig:värigraafiesim}(b) has two edges with the colour $v_7$ between non-codewords, but the codeword $v_7$ is not non-swappable.
Indeed, the vertices $v_6$ and $v_7$ can be freely swapped in any LD-code containing one but not the other as they are twins, that is, they have the same closed neighbourhood in $G$.

\begin{corollary}
\label{lemma:atleast2edges-forced}
    Let $G$ be a connected nontrivial graph and let $S$ be a minimum locating-dominating code in $G$.
    If $v \in S$ is min-forced, then there are at least 2 edges with the colour $v$ in the graph $\graphminus{G_S}{S}$. 
\end{corollary}
\begin{proof}
    Minimum LD-codes are minimal.
    Min-forced vertices are non-swappable with respect to every minimum LD-code in $G$.
    Therefore, by Lemma~\ref{lemma:atleast2edges}, there are at least 2 edges with the colour $v$ in the graph $\graphminus{G_S}{S}$  for each min-forced $v\in S$.
\end{proof}

Notice that our new (extended) colour graph always has at least two edges of the same colour in $\graphminus{G_S}{S}$ for every min-forced vertex.
This allows us to utilise a fruitful approach discussed in \cite{hernando2018locating}.
Although our graph $G_S$ is different from the one in \cite{hernando2018locating}, we can use, due to the similar properties in $G_S-S$ of our colour graph, the same reasoning as in Lemma 1 of \cite{hernando2018locating} to prove the analogous result below in Lemma~\ref{lemma:2_edges_cacti}.
Indeed, it is straightforward to check that the arguments to prove Lemma 1 in  \cite{hernando2018locating} relies on the properties (see Proposition 5 in \cite{hernando2018locating}) of their colour graph that hold in our colour graph as well
(due to (iii), (vii) and (viii) stated in Lemma~\ref{lemma:G_S-properties}).
For the following lemma, we define that a \emph{cactus} is a graph such that no two cycles share an edge. 

\begin{lemma}
    \label{lemma:2_edges_cacti}
    Let $S$ be a locating-dominating code in a nontrivial graph $G$ and let $S' \subseteq S$.
    Consider a subgraph $H$ of $\graphminus{G_S}{S}$ induced by a set of edges containing exactly two edges with colour $u$ for each $u \in S'$. Then all the connected components of $H$ are cacti.
\end{lemma}

To illustrate the previous lemma,
see Figure~\ref{fig:p5esim}(a) for an example of a graph with four min-forced vertices.
The cactus graph in Figure~\ref{fig:p5esim}(b) is the induced subgraph of the colour graph with two edges of each colour.

\begin{lemma}[Lemma 2 in \cite{hernando2018locating}]
    \label{lemma:cactus_bound}
    Let $H$ be a bipartite graph such that $\abs{V(H)} \geq 4$ and all its connected components are cacti. Then $\abs{V(H)} \geq \frac{3}{4} \abs{E(H)} + \cc(H)$, where $\cc(H)$ is the number of connected components in $H$.
\end{lemma}

The graph in Figure~\ref{fig:p5esim}(b) reaches the bound of Lemma~\ref{lemma:cactus_bound}. The next theorem gives us upper bounds for the number of min-forced vertices. In particular, it shows that the maximum ratio between the number of min-forced vertices and the order of a nontrivial graph is $2/5$, which is also shown to be optimal in Example~\ref{Ex_upper_bound_attained}.

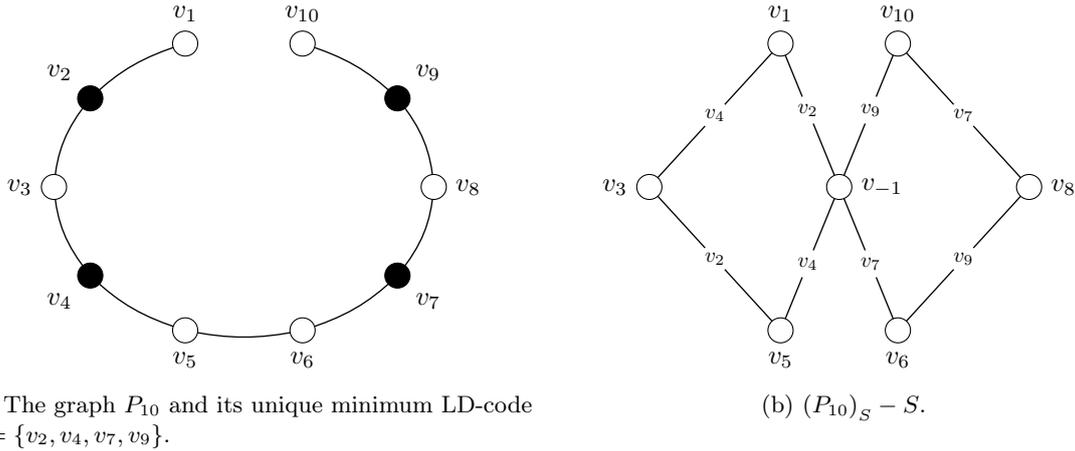
\begin{figure}
     \centering
     \begin{subfigure}[t]{0.47\textwidth}
        \centering
        \begin{tikzpicture}[yscale=2, xscale=2.5, rotate=108]

    \renewcommand*{\EdgeLineWidth}{ 0.5pt}
    \SetVertexMath
    \SetVertexLabelOut

    {
    \SetVertexNoLabel
    \tikzstyle{VertexStyle}=[minimum size=1pt,inner sep=0pt]
    \Vertices{circle}{1,2,3,4,5,6,7,8,9,0} 
    }

    \tikzstyle{EdgeStyle}=[bend right=14]
    \Edge(1)(2)
    \Edge(2)(3)
    \Edge(3)(4)
    \Edge(4)(5)
    \Edge(5)(6)
    \Edge(6)(7)
    \Edge(7)(8)
    \Edge(8)(9)
    \Edge(9)(0)

    \tikzstyle{VertexStyle}=[circle,draw, fill=black]
    \SO[unit=0, Lpos=-45+180](2){v_2}
    \SO[unit=0, Lpos=45+180](4){v_4}
    \SO[unit=0, Lpos=135+180](7){v_7}
    \SO[unit=0, Lpos=-135+180](9){v_9}

    \tikzstyle{VertexStyle}=[circle,draw, fill=white]
    \SO[unit=0, Lpos=-90+180](1){v_1}
    \SO[unit=0, Lpos=0+180](3){v_3}
    \SO[unit=0, Lpos=90+180](5){v_5}
    \SO[unit=0, Lpos=90+180](6){v_6}
    \SO[unit=0, Lpos=180+180](8){v_8}
    \SO[unit=0, Lpos=-90+180](0){v_{10}}

\end{tikzpicture}
        \caption{The graph $P_{10}$ and its unique minimum LD-code $S =\{v_2, v_4, v_7, v_9\}$.}
     \end{subfigure}
     \hfill
     \begin{subfigure}[t]{0.47\textwidth}
        \centering
        \begin{tikzpicture}[yscale=2, xscale=2.5, rotate=108]

    \renewcommand*{\EdgeLineWidth}{ 0.5pt}
    \SetVertexMath
    \SetVertexLabelOut

    {
    \SetVertexNoLabel
    \tikzstyle{VertexStyle}=[minimum size=1pt,inner sep=0pt]
    \Vertices{circle}{1,2,3,4,5,6,7,8,9,0} 
    }

    \tikzstyle{VertexStyle}=[circle,draw, fill=white]
    \Vertex[x=0, y=0]{\auxiliary}

    \tikzstyle{LabelStyle}=[fill=white, circle,minimum size=12pt,inner sep=1pt,scale=.8]

    \tikzstyle{EdgeStyle}=[black]
    \Edge[label=$v_2$](\auxiliary)(1)
    \Edge[label=$v_4$](1)(3)
    \Edge[label=$v_2$](3)(5)
    \Edge[label=$v_4$](5)(\auxiliary)
    \Edge[label=$v_7$](\auxiliary)(6)
    \Edge[label=$v_9$](6)(8)
    \Edge[label=$v_7$](8)(0)
    \Edge[label=$v_9$](0)(\auxiliary)

    \tikzstyle{VertexStyle}=[circle,draw, fill=white]
    \SO[unit=0, Lpos=90](1){v_1}
    \SO[unit=0, Lpos=180](3){v_3}
    \SO[unit=0, Lpos=-90](5){v_5}
    \SO[unit=0, Lpos=-90](6){v_6}
    \SO[unit=0, Lpos=0](8){v_8}
    \SO[unit=0, Lpos=90](0){v_{10}}

\end{tikzpicture}
        \caption{$\graphminus{{(P_{10})}_S}{S}$.}
     \end{subfigure}

    \caption{The illustration using the path $P_{10}$.}
    \label{fig:p5esim}
\end{figure}

\begin{theorem}\label{theo:forced-40-bound}
    If $G$ is a connected nontrivial graph with $k \geq 1$ min-forced vertices, then \[k \leq \frac{2}{3}\left(n -\gamma^{LD}(G)\right)\] and, in particular, $k \leq \frac{2}{5}n$. Furthermore, for all nontrivial graphs $G$ with min-forced vertices, we have $\gamma^{LD}(G) \leq n-3$.
\end{theorem}
\begin{proof}
Let $S \subseteq V(G)$ be a minimum locating-dominating code in a graph $G$ with $k \geq 1$ min-forced vertices.
Denote the non-codewords $U = V(G) \setminus S$. 
By Lemma~\ref{lemma:atleast2edges-forced}, there are at least $2$ edges with the colour $v$ in $G_S[U \cup \{\auxiliary\}] = G_S - S$ for each min-forced vertex $v \in S$.
Let $E' \subseteq E(G_S - S)$ be a set of edges consisting of exactly $2$ edges of colour $v$ for each min-forced $v$.
Since $k \geq 1$ and the edges with the same colour do not share an endpoint by Lemma~\ref{lemma:G_S-properties}(iii), there are at least 4 vertices in $G_S[E']$.
Having 4 vertices in  $G_S[E']$ implies that $\gamma^{LD}(G) \leq (n+1) - 4 = n-3$.
By Lemmas~\ref{lemma:2_edges_cacti} and \ref{lemma:G_S-properties}(vii), the connected components of $G_S[E']$ are bipartite cacti. 
Therefore, by Lemma~\ref{lemma:cactus_bound}, we have $\abs{U} + 1 \geq \abs{V(G_S[E'])}  \geq \frac{3}{4} \abs{E(G_S[E'])} + \cc(G_S[E']) \geq \frac{3}{2}k + 1$. 
Then, using $k \leq \gamma^{LD}(G)$ we get
\begin{align*}
     \frac{3}{2} k \leq \abs{U} =n - \gamma^{LD}(G) \leq n - k,
\end{align*}
from which we get the bounds $k \leq \frac{2}{3}\left(n -\gamma^{LD}(G)\right)$ and  $k \leq \frac{2}{5}n$.
\end{proof}

In what follows, we show that the upper bound $k \leq \frac{2}{3}\left(n -\gamma^{LD}(G)\right)$ can be attained for all even positive integers $k$. For this purpose, we need some auxiliary results. We first recall the minimum cardinality of LD-codes in paths, which is known due to Slater \cite{slater1987domination}.
\begin{lemma}[\cite{slater1987domination}]
\label{lemma:path-bound}
    We have $\gamma^{LD}(P_n) = \ceil{\frac{2n}{5}}$.
\end{lemma}

Furthermore, we define a \emph{weakened} version of LD-codes in paths, which turns out to be useful also in Section~\ref{pathsection}.
\begin{definition}
    A code $S \subseteq V(P_n)$ is an \emph{LD*-code}, if for all distinct non-codewords $v_i, v_j \in V(P_n) \setminus (S \cup \{v_n\})$ their $I$-sets are nonempty and
    \(I(v_i) \neq I(v_j)\).
    In other words, $S$ is a code for which the locating-dominating property is satisfied for all but the last vertex $v_n$.  The cardinality of a minimum LD*-code in $P_n$ is denoted by $\gamma^{LD^{*}}(P_n)$.
\end{definition}

The cardinality $\gamma^{LD^{*}}(P_n)$ is determined in the following theorem.
\begin{theorem}\label{theo:LD*opt}
    We have $\gamma^{LD^{*}}(P_n) = \ceil{\frac{2(n-1)}{5}}$.
\end{theorem}
\begin{proof}
    Since all LD-codes in $P_{n-1}$ are LD*-codes in $P_n$, we have $\gamma^{LD^{*}}(P_n) \leq \gamma^{LD}(P_{n-1}) = \ceil{\frac{2(n-1)}{5}}$. 
    
    Let $S^*$ be an LD*-code in $P_n$.
    If $v_n \notin S^*$, then $S^*$ is an LD-code in $P_{n-1}$.
    If $v_n \in S^*$, then $S^*[v_n \leftarrow v_{n-1}]$ is an LD-code in $P_{n-1}$ with the same or smaller cardinality as $S^*$.
    It follows that $\gamma^{LD^{*}}(P_n) \geq \gamma^{LD}(P_{n-1}) = \ceil{\frac{2(n-1)}{5}}$.
\end{proof}

Now we are ready to present in the following example a family of graphs attaining the upper bound $k \leq \frac{2}{3}\left(n -\gamma^{LD}(G)\right)$ for all even positive integers $k$. It should be noted that the presented graphs $G_{s,t}$ are trees and, hence, their min-forced vertices could be determined based on the characterization given by Blidia and Lounes~\cite{Blidia2009minforcedtrees}. However, as their characterization of min-forced vertices in trees is somewhat technical, we have decided to present a simple independent proof based on Theorem~\ref{theo:forced-characterization}.

\begin{example} \label{Ex_upper_bound_attained}
    Assume that $s$ and $t$ are positive integers. 
    Let $G_{s,t}$ be a graph formed by a path $P_s$, the vertices of which are denoted by $v_1, v_2, \ldots, v_s$, and pendant vertices $u_1, u_2, \ldots, u_t$, which are attached to $v_s$; occasionally, $G_{s,t}$ has been called the \emph{broom graph} (due to an apparent visual resemblance). 
    In what follows, we first give a lower bound on $\gamma^{LD}(G_{s,t})$. 
    For this purpose, let $s \geq 4$ and $S$ be an LD-code in $G_{s,t}$. 
    Now we can make the following observations:
	\begin{itemize}
            \item If $v_s \notin S$, then the vertices $u_1, u_2, \ldots, u_t$ belong to $S$ since otherwise $I(u_i) = \emptyset$ for some $i \in \{1, \dots, t\}$, and furthermore, $\{v_{s-1}, v_{s-2}\} \cap S \neq \emptyset$, since otherwise $I(v_{s-1}) = \emptyset$.
  
            \item If $v_s \in S$, then at least $t-1$ of the (open) twins $u_i$ belong to $S$, say $\{u_1, u_2, \ldots, u_{t-1}\} \subseteq S$. Furthermore, if $u_t \notin S$, then $v_{s-1} \in S$ or $v_{s-2} \in S$ as otherwise $I(v_{s-1}) = I(u_{t})$.
	\end{itemize}
	Thus, in conclusion, $\abs{S \cap \{v_{s-2}, v_{s-1}, v_{s}, u_1, u_2, \ldots, u_t\}} \geq t+1$. Moreover, $S \cap \{v_1, v_2, \ldots, v_{s-3}\}$ is an LD*-code in $P_{s-3}$. Hence, we have
	\begin{equation} \label{Eq_lower_bound_on_Gst}
		\gamma^{LD}(G_{s,t}) \geq \gamma^{LD^*}(P_{s-3}) + t + 1 \text.
	\end{equation}
	
	Consider then the graph $G_{5m+4,t}$ for a nonnegative integer $m$. It is straightforward to verify that the set $\{v_{5\ell+2} \mid \ell = 0, 1, \ldots, m\} \cup \{v_{5\ell+4} \mid \ell = 0, 1, \ldots, m\} \cup \{u_1, u_2, \ldots, u_{t-1}\}$ with $2m + t + 1$ vertices is locating-dominating in $G_{5m+4,t}$. Therefore, by~\eqref{Eq_lower_bound_on_Gst} and Theorem~\ref{theo:LD*opt}, we have $\gamma^{LD}(G_{5m+4,t}) = 2m + t + 1$. In what follows, we show that the vertices of $\{v_{5\ell+2} \mid \ell = 0, 1, \ldots, m\} \cup \{v_{5\ell+4} \mid \ell = 0, 1, \ldots, m\}$ are min-forced in $G_{5m+4,t}$:
	\begin{itemize}
		\item For $\ell = 0,1, \ldots, m-1$, the vertex $v_{5\ell+2}$ is min-forced by Theorem~\ref{theo:forced-characterization}(ii) since
		\[
		\begin{split}
			\gamma^{LD}(G_{5m+4,t} - v_{5\ell+2}) &= \gamma^{LD}(P_{5\ell+1}) + \gamma^{LD}(G_{5m+4-(5\ell+2),t}) \\
            &\geq \left\lceil \frac{2(5\ell+1)}{5} \right\rceil + \gamma^{LD^*}(P_{5(m-\ell)-1}) + t + 1 \\
            &= 2\ell + \left\lceil \frac{2(5(m-\ell)-2)}{5} \right\rceil  + t + 2 \\
            &= 2\ell + 2(m-\ell) + t + 2 \\
            &> \gamma^{LD}(G_{5m+4,t}) \text,
		\end{split}
		\]
        where we apply Theorem~\ref{theo:LD*opt} to compute $\gamma^{LD^*}(P_{5(m-\ell)-1})$.
		
		\item Similarly, for $\ell = 0,1, \ldots, m-1$, the vertex $v_{5\ell+4}$ can be shown to be min-forced by Theorem~\ref{theo:forced-characterization}(ii).
		
		\item The vertex $v_{5m+2}$ is min-forced by Theorem~\ref{theo:forced-characterization}(ii) since 
		\[
        \begin{split}
            \gamma^{LD}(G_{5m+4,t} - v_{5m+2}) &= \gamma^{LD}(P_{5m+1}) + \gamma^{LD}(K_{1,t+1}) = \left\lceil \frac{2(5m+1)}{5} \right\rceil + t + 1 \\ &= 2m + t + 2 > \gamma^{LD}(G_{5m+4,t}) \text.
        \end{split}
		\]
		Indeed, it is well-known that for the star $K_{1,t+1}$, we have $\gamma^{LD}(K_{1,t+1}) = t + 1$.
		
		\item The vertex $v_{5m+4}$ is min-forced by Theorem~\ref{theo:forced-characterization}(ii) since $G_{5m+4,t} - v_{5m+4}$ consists of a path $P_{5m+3}$ and $t$ isolated vertices implying that 
		\[
		\gamma^{LD}(G_{5m+4,t} - v_{5m+4}) = \left\lceil \frac{2(5m+3)}{5} \right\rceil + t = 2m + t + 2 \text.
		\]
	\end{itemize}
	
	Thus, in conclusion, we have $n = \abs{V(G_{5m+4,t})} = 5m + t + 4$, $\gamma^{LD}(G_{5m+4,t}) = 2m + t + 1$ and there are $2m+2$ min-forced vertices in $G_{5m+4,t}$. Hence, the graph $G_{5m+4,t}$ attains the upper bound of Theorem~\ref{theo:forced-40-bound} as
	\[
	\frac{2}{3}\left( n - \gamma^{LD}(G_{5m+4,t}) \right) = \frac{2}{3}\left( (5m + t + 4) - (2m + t + 1) \right) = 2m + 2\text.
	\]
\end{example}
Notice that if $t = 1$ in the previous example, then $G_{5m+4,1} = P_{5(m+1)}$ and the number of min-forced vertices is equal to $\gamma^{LD}(P_{5(m+1)}) = 2 \cdot 5(m+1) /5 = 2m+2$. In other words, the minimum locating-dominating code in $P_{5(m+1)}$ is unique. 
The number of minimum LD-codes in paths is further discussed in the following section in Theorem~\ref{theo:number-of-codes}, where another approach is given for the uniqueness of the code in $P_{5(m+1)}$. 
In conclusion, Theorem~\ref{theo:forced-40-bound} and Example~\ref{Ex_upper_bound_attained} are combined in the following corollary.
\begin{corollary} \label{Cor_max_number_of_min-forced}
    If $G$ is a connected nontrivial graph with $k \geq 1$ min-forced vertices, then \[k \leq \frac{2}{3}\left(n -\gamma^{LD}(G)\right) \quad \text{and}\quad k \leq \frac{2}{5}n.\] Moreover, there exist graphs for which the equalities hold. In particular, the maximum ratio between the number of min-forced vertices and the order of a connected nontrivial graph is $2/5$.
\end{corollary}

\section{The number of different minimum locating-dominating codes in paths}\label{pathsection}

In the previous section, we noticed that the LD-code formed by the min-forced vertices is unique in $P_{5k}$. 
A natural extension of this observation is the question of how many minimum locating-dominating codes there are in the path $P_n$ in general.
To address this question, 
we denote the number of minimum locating-dominating codes in $P_n$ by $\koodit(n)$ and the number of LD*-codes in $P_n$ with cardinality $k$ by  $\apu(n, \koko)$.
In the following theorem, we calculate the number of different minimum locating-dominating codes for all
paths, that is, we determine $\koodit(n)$ for lengths $n \geq 5$.
For completeness, we mention that the following first values of $C(n)$ are straightforward to verify: 
\begin{align*}
    \koodit(1) = 1,\ 
    \koodit(2) = 2,\ 
    \koodit(3) = 3,\ \text{and }
    \koodit(4) = 4.
\end{align*}

\begin{theorem}\label{theo:number-of-codes}
    The number of minimum locating-dominating codes in the path $P_{5\m + r}$ for each $r \in \{0, \dots, 4\}$ and $m \geq 1$ are the following: 

        \begin{itemize}
        \item[($M_0$)] 
        $\koodit(5\m) = 1$
        \item[($M_1$)]  
        $\koodit(5\m + 1) = \frac{1}{6}\m^3 + 2\m^2 + \frac{29}{6}\m + 1$  

        \item[($M_2$)]  
        $\koodit(5\m + 2) = \m + 2$
        \item[($M_3$)]  

        $\koodit(5\m + 3) = \frac{1}{24}\m^4 + \frac{11}{12}\m^3 + \frac{131}{24}\m^2 + \frac{103}{12}\m + 3$
        \item[($M_4$)]  
        $\koodit(5\m + 4)= \frac{1}{2}\m^2 + \frac{7}{2}\m + 4$.

    \end{itemize}
\end{theorem}

\begin{proof}
Let $S$ be an LD-code in $P_n$, where $n \geq 5$.
Anticipating the recursion formula in Equation~\eqref{eq:koodien-lkm},
the code $S$ is classified into the following three categories based on the pattern of codewords in the final few vertices of the path:

\begin{itemize}
    \item  $v_n \in S$,
    
    \item  $v_n\notin S$ and $v_{n-1}, v_{n-2} \in S$, or
    
    \item $v_n, v_{n-2} \notin S$ and $v_{n-1}, v_{n-3} \in S$. 
\end{itemize}
These patterns are illustrated in Figure~\ref{fig:polun-loppu}.
In what follows, we discuss the  minimum LD-codes in each category.
\begin{itemize}
    \item Suppose first that $v_n \in S$.
    Clearly, the set $S \cap \{v_1, \dots, v_{n-1}\}$ is an LD*-code in $P_{n-1}$ with cardinality $\gamma^{LD}(P_n)-1$. On the other hand, if $A$ is an LD*-code in $P_{n-1}$ with cardinality $\gamma^{LD}(P_n)-1$, then $A\cup\{v_n\}$ is a minimum LD-code in $P_n$. This is clear if $v_{n-1}\in A$, so assume that $v_{n-1}\notin A$. Now the vertex $v_{n-1}$ is dominated and has a unique $I$-set due to $v_n$. Moreover, each $A$ gives a different LD-code for $P_n.$
    
    \item If $v_n \notin S$, then necessarily $v_{n-1} \in S$. Now either $v_{n-2} \in S$ or $v_{n-2} \notin S$.
    \begin{itemize}
        \item Let us first suppose that $v_{n-2} \in S$. Now the set $S \cap \{v_1, \dots, v_{n-3}\}$ is an LD*-code in $P_{n-3}$ with cardinality $\gamma^{LD}(P_n)-2$. Conversely, 
        if $A$ is an LD*-code in $P_{n-3}$ with cardinality $\gamma^{LD}(P_n)-2$, then $A\cup\{v_{n-2},v_{n-1}\}$ is a minimum LD-code in $P_n$. Indeed if $v_{n-3}\in A$, then we are immediately done, and otherwise  ($v_{n-3}\notin A$) the vertex $v_{n-3}$ is dominated by $v_{n-2}$ and its $I$-set contains $v_{n-2}$, which no other non-codeword does.
        
        \item Suppose next that $v_{n-2} \notin S$. To distinguish the non-codewords $v_n$ and $v_{n-2}$, we need to have $v_{n-3} \in S$. The set $S \cap \{v_1, \dots, v_{n-4}\}$ is an LD*-code in $P_{n-4}$ with cardinality $\gamma^{LD}(P_n)-2$. On the other hand, if $A$ is an LD*-code in $P_{n-4}$ with cardinality $\gamma^{LD}(P_n)-2$, then $A\cup\{v_{n-3},v_{n-1}\}$ is a minimum LD-code in $P_n$, since (if $v_{n-4}\notin A$) the vertex $v_{n-4}$  is dominated by $v_{n-3}$ and distinguished from other $I$-sets of $v_i$ ($i\le n-5)$ by $v_{n-3}$.

    \end{itemize}
\end{itemize}

As seen above, the three categories described above partition the set of minimum LD-codes in $P_n$, and hence each LD-code is counted exactly once by counting the number of appropriate LD*-codes in shorter paths.
Thus, recalling Lemma~\ref{lemma:path-bound}, we get the following formula for the number of minimum LD-codes in $P_n$ (based on the patterns illustrated in Figure~\ref{fig:polun-loppu}): 
\begin{equation}
\label{eq:koodien-lkm}
\begin{split}
    \koodit(n) &= \apu(n - 1, \ceil{{2n}/{5}} - 1)\\
            &+ \apu(n - 3, \ceil{{2n}/{5}} - 2)\\
            &+ \apu(n - 4, \ceil{{2n}/{5}} - 2).
\end{split}     
\end{equation}

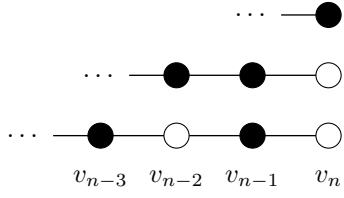
\begin{figure}
    \centering
    \begin{tikzpicture}[yscale=0.8, xscale=1, rotate=180]

    \renewcommand*{\EdgeLineWidth}{ 0.5pt}
    \SetVertexMath
    \SetVertexLabelOut

    {
    \SetVertexNoLabel

    \tikzstyle{VertexStyle}=[circle,draw, fill=black]
    \Vertex[x=0, y=0]{00}

    \Vertex[x=1, y=1]{11}
    \Vertex[x=2, y=1]{21}

    \tikzstyle{VertexStyle}=[circle,draw, fill=white]

    \Vertex[x=0, y=1]{01}

    }

    \tikzstyle{VertexStyle}=[circle,draw, fill=black]

    \Vertex[x=3, y=2, L=v_{n-3}, Lpos=-90, Ldist=5]{32}
    \Vertex[x=1, y=2, L=v_{n-1}, Lpos=-90, Ldist=5]{12}
    
    \tikzstyle{VertexStyle}=[circle,draw, fill=white]
    
    \Vertex[x=0, y=2, L=v_{n}, Lpos=-90, Ldist=5]{02}
    \Vertex[x=2, y=2, L=v_{n-2}, Lpos=-90, Ldist=5]{22}

{
    \tikzstyle{VertexStyle}=[fill=white,circle]
    \SetVertexLabelIn

    \Vertex[x=4, y=2, L=\dots]{42}
    \Vertex[x=3, y=1, L=\dots]{31}
    \Vertex[x=1, y=0, L=\dots]{10}

}

    \tikzstyle{EdgeStyle}=[bend right=0]
    \Edge(00)(10)
    
    \Edge(01)(11)
    \Edge(11)(21)
    \Edge(21)(31)
    
    \Edge(02)(12)
    \Edge(12)(22)
    \Edge(22)(32)
    \Edge(32)(42)

\end{tikzpicture}
    \caption{The different categories of LD-codes illustrated. The black and white vertices represent the codewords and non-codewords, respectively.
    }
    \label{fig:polun-loppu}
\end{figure}

For the values of $\apu(n,\koko)$, we get the following recurrence relation:
\begin{equation}\label{eq:apukoodien-lkm}
    \begin{split}
        \apu(n, k) &= \apu(n - 1, \koko - 1)\\
            &+ \apu(n - 2, \koko - 1)\\
            &+ \apu(n - 4, \koko - 2) \\
            &+ \apu(n - 5, \koko - 2).
    \end{split}
\end{equation}

We now reason why the recurrence relation holds. 
Each LD*-code $S$ in $P_n$ (where $n \geq 6$) with $\abs{S} = \koko$ is, as will be seen, in exactly one of the following four categories, illustrated in Figure~\ref{fig:polun-loppu*}:
\begin{enumerate}[(i)]
    \item  $v_n \in S$,
    
    \item  $v_n\notin S$ and $v_{n-1}\in S$,
    
    \item $v_n, v_{n-1} \notin S$ and $v_{n-2},v_{n-3} \in S$, or
    
    \item $v_n, v_{n-1}, v_{n-3} \notin S$ and $v_{n-2},v_{n-4} \in S$. 
\end{enumerate}

Next, we consider each type of LD*-code. First we notice that $S\cap \{v_1, \dots, v_{n-j}\} $ is an LD*-code of cardinality $\koko-1$ in $P_{n-j}$ where $j=1$ in the category (i) and  $j=2$ in the category (ii). Similarly, $S\cap \{v_1, \dots, v_{n-j}\} $ is an LD*-code of cardinality $\koko-2$ in $P_{n-j}$ where $j=4$ in the category (iii) and $j=5$ in the category (iv). Next we consider different endings of LD*-codes in $P_n$ and show  that a suitable LD*-code in a shorter path provides an LD*-code in $P_n$.

\begin{itemize}
    \item Let first $v_n \in S$.  If $A$ is an LD*-code with cardinality $\koko-1$ in $P_{n-1}$, then $A\cup\{v_n\}$ is an LD*-code of cardinality $k$ in $P_n$ as the $I$-set of $v_{n-1}$ is nonempty and unique (the reasoning is needed only if $v_{n-1}\notin A$).
    
    \item Suppose then that $v_n \notin S$ and $v_{n-1} \in S$. If $A$ is an LD*-code with cardinality $\koko-1$ in $P_{n-2}$, then $A\cup\{v_{n-1}\}$ is an LD*-code of cardinality $k$ in $P_n$. Indeed, if $v_{n-2}\in A$, we are done, and if $v_{n-2}\notin A$, the set $I(v_{n-2})$ contains at least $v_{n-1}$, so it is nonempty and necessarily distinct from $I$-sets other than possibly $I(v_n)$, which is allowed in an LD*-code of $P_n$. 
    
    \item If $v_n, v_{n-1} \notin S$, then we must have $v_{n-2} \in S$ in order for $v_{n-1}$ to be dominated. (Recall that the final vertex need not be dominated in an LD*-code.) Then, in order to have distinct $I$-sets for the vertices $v_i \notin S$ where $i \in \{n-3, n-1\}$, either $v_{n-3} \in S$ or $v_{n-3} \notin S$ and $v_{n-4} \in S$.
    \begin{itemize}
        \item Let first $v_{n-3} \in S$.  If $A$ is an LD*-code with cardinality $\koko-2$ in $P_{n-4}$, then $A\cup\{v_{n-3},v_{n-2}\}$ is an LD*-code of cardinality $k$ in $P_n$, since (if $v_{n-4}\notin A)$ the codeword $v_{n-3}$ dominates and distinguishes $v_{n-4}$.
        \item Finally, let $v_{n-3} \notin S$ and $v_{n-4} \in S$. If $A$ is an LD*-code with cardinality $\koko-2$ in $P_{n-5}$, then $A\cup\{v_{n-4},v_{n-2}\}$ is an LD*-code of cardinality $k$ in $P_n$, since  (if $v_{n-5}\notin A$) the vertex $v_{n-5}$ is dominated and distinguished by $v_{n-4}$.
    \end{itemize}
\end{itemize}
As we see above, the four categories (i)--(iv) are the only possible patterns of codewords in an LD*-code at the end of the path $P_n$, and they are mutually exclusive, so each LD*-code in $P_n$ with cardinality $k$ is counted exactly once in the recursion formula~\eqref{eq:apukoodien-lkm}.

Theorem~\ref{theo:LD*opt} gives the boundary condition
\begin{equation} \label{eq:LD*bound}
    \apu(n, \koko) = 0, \text{ when } \koko <  \ceil{\frac{2(n-1)}{5}}.
\end{equation}
    
\begin{figure}
    \centering
        \begin{tikzpicture}[yscale=0.8, xscale=1, rotate=180]

    \renewcommand*{\EdgeLineWidth}{ 0.5pt}
    \SetVertexMath
    \SetVertexLabelOut

    {
    \SetVertexNoLabel

    \tikzstyle{VertexStyle}=[circle,draw, fill=black]
    \Vertex[x=0, y=0]{00}

    \Vertex[x=1, y=1]{11}

    \Vertex[x=2, y=2]{22}
    \Vertex[x=3, y=2]{32}
    
    \tikzstyle{VertexStyle}=[circle,draw, fill=white]
    
    \Vertex[x=0, y=2]{02}
    \Vertex[x=1, y=2]{12}
    \Vertex[x=0, y=1]{01}

    }

    \tikzstyle{VertexStyle}=[circle,draw, fill=black]

    \Vertex[x=4, y=3, L=v_{n-4}, Lpos=-90, Ldist=5]{43}
    \Vertex[x=2, y=3, L=v_{n-2}, Lpos=-90, Ldist=5]{23}
    
    \tikzstyle{VertexStyle}=[circle,draw, fill=white]
    
    \Vertex[x=0, y=3, L=v_{n}, Lpos=-90, Ldist=5]{03}
    \Vertex[x=1, y=3, L=v_{n-1}, Lpos=-90, Ldist=5]{13}
    \Vertex[x=3, y=3, L=v_{n-3}, Lpos=-90, Ldist=5]{33}

{
    \tikzstyle{VertexStyle}=[fill=white,circle]
    \SetVertexLabelIn
    
    \Vertex[x=5, y=3, L=\dots]{53}
    \Vertex[x=4, y=2, L=\dots]{42}
    \Vertex[x=2, y=1, L=\dots]{21}
    \Vertex[x=1, y=0, L=\dots]{10}

}

    \tikzstyle{EdgeStyle}=[bend right=0]
    \Edge(00)(10)
    
    \Edge(01)(11)
    \Edge(11)(21)
    
    \Edge(02)(12)
    \Edge(12)(22)
    \Edge(22)(32)
    \Edge(32)(42)
    \Edge(03)(13)    
    \Edge(13)(23)
    \Edge(23)(33)
    \Edge(33)(43)
    \Edge(43)(53)

\end{tikzpicture}
    \caption{The four categories of LD*-codes illustrated.}
    \label{fig:polun-loppu*}
\end{figure}
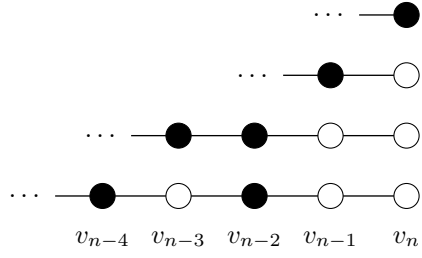

We will first prove the following values of $\apu(n, \koko)$ in this order and refer back to them as needed:
\begin{equation} \label{eq:5k+1}
     \apu(5\K +1, 2\K) = 1,
\end{equation}
\begin{equation} \label{eq:5k+3}
    \apu(5\K +3, 2\K+1) = \K + 2,
\end{equation}
\begin{equation}\label{eq:5k}
        \apu(5\K, 2\K) =  \frac{1}{2}\K^2 + \frac{5}{2}\K + 1,
\end{equation}
\begin{equation}\label{eq:5k+2}
    \apu(5\K+2, 2\K+1) =  \frac{1}{6}\K^3 + 2\K^2 + \frac{29}{6}\K + 2,
\end{equation}
\begin{equation}\label{eq:5k+4}
        \apu(5\K + 4, 2\K + 2) =  \frac{1}{24}\K^4 + \frac{11}{12}\K^3 + \frac{131}{24}\K^2 + \frac{115}{12}\K + 5.
\end{equation}

In what follows, we prove Equations~\eqref{eq:5k+1}--\eqref{eq:5k+4}.

\noindent {\bf Equation~\eqref{eq:5k+1}}: First, we calculate $\apu(5\m + 1, 2\m)=1$.
By~\eqref{eq:apukoodien-lkm},
\begin{align*}
    \apu(5\m +1, 2\m ) &= \apu(5\m +1 - 1, 2\m - 1) 
            + \apu(5\m +1 - 2, 2\m - 1)\\
            &+ \apu(5\m +1 - 4, 2\m - 2) 
            + \apu(5\m +1 - 5, 2\m - 2), 
\end{align*}
where the first three terms are eliminated by the bound~\eqref{eq:LD*bound}, 
because \(2\m - 1 < \ceil{\frac{2(5\m-1)}{5}}\), \(2\m - 1 < \ceil{\frac{2(5\m-1-1)}{5}}\), and \(2\m - 2 < \ceil{\frac{2(5\m-3-1)}{5}}\).
We are left with
$$\apu(5\m +1, 2\m) = \apu(5(\m-1) +1, 2(\m - 1)),$$
where through iteration we get the base case $\apu(1,0)$, that is, the number of LD*-codes in $P_1$ with no vertices.
There is one such code, namely $\emptyset$.
Therefore,
\begin{equation*}
     \apu(5\m +1, 2\m) = \apu(1, 0) = 1.
\end{equation*}

\noindent {\bf Equation~\eqref{eq:5k+3}}: Next, we calculate $\apu(5\m +3, 2\m+1) = \K + 2$.
By~\eqref{eq:apukoodien-lkm},
\begin{align*}
    \apu(5\m +3, 2\m + 1) &= \apu(5\m +3 - 1, 2\m + 1 - 1)
            + \apu(5\m +3 - 2, 2\m + 1 - 1)\\
            &+ \apu(5\m +3 - 4, 2\m + 1 - 2)
            + \apu(5\m +3 - 5, 2\m + 1 - 2), 
\end{align*}
where the first and third terms are eliminated by~\eqref{eq:LD*bound} and the value of the second term is known by Equation~\eqref{eq:5k+1}.
We are left with the recurrence relation
$$\apu(5\m +3, 2\m +1) = 1 + \apu(5(\m-1) +3, 2(\m - 1)+1),$$
where after $\m$ iterations we arrive at $\apu(3,1)$, the number of LD*-codes in $P_3$ with one vertex.
There are two such codes, namely $\{v_1\}$ and $\{v_2\}$.
Thus, we obtain
\begin{equation*}
    \apu(5\K +3, 2\K+1) = \K\cdot1 + \apu(3,1) = \K + 2.
\end{equation*}

\noindent {\bf Equation~\eqref{eq:5k}}: Let us prove that $\apu(5\m, 2\m) = \frac{1}{2}\K^2 + \frac{5}{2}\K + 1$.
By~\eqref{eq:apukoodien-lkm},
\begin{align*}
    \apu(5\m, 2\m) &= \apu(5\m - 1, 2\m - 1) 
            + \apu(5\m - 2, 2\m - 1)\\
            &+ \apu(5\m - 4, 2\m - 2) 
            + \apu(5\m - 5, 2\m - 2), 
\end{align*}
where the first term is eliminated by the bound~\eqref{eq:LD*bound} and the values of the second and third terms are known by Equations~\eqref{eq:5k+3} and \eqref{eq:5k+1}.
We are left with the recurrence relation
$$\apu(5\m, 2\m) = (\m + 1) +  1 + \apu(5(\m-1), 2(\m - 1)),$$
where, after $m-1$ iterations, we arrive at $\apu(5,2)$, the number of LD*-codes in $P_5$ with two elements.
There are 4 such codes: $\{v_1, v_3\}$, $\{v_1, v_4\}$, $\{v_2, v_3\}$ and $\{v_2, v_4\}$.
Thus, we obtain
\begin{equation*}
\begin{split}
        \apu(5\K, 2\K) &= (\K + 1) +  1 + \apu(5(\K-1), 2(\K - 1)) \\
        &= \sum_{i = 2}^{\K}( i + 2) + \apu(5,2) \\
        &= \frac{1}{2}\K^2 + \frac{5}{2}\K + 1.
\end{split}
\end{equation*}

\noindent {\bf Equation~\eqref{eq:5k+2}}: Let us show that $\apu(5\m+2, 2\m+1) = \frac{1}{6}\K^3 + 2\K^2 + \frac{29}{6}\K + 2$.
By~\eqref{eq:apukoodien-lkm},
\begin{align*}
    \apu(5\m+2, 2\m+1) &= \apu(5\m +2 - 1, 2\m + 1 - 1) 
            + \apu(5\m +2 - 2, 2\m + 1 - 1)\\
            &+ \apu(5\m +2 - 4, 2\m + 1 - 2) 
            + \apu(5\m +2 - 5, 2\m + 1 - 2), 
\end{align*}
where the first three terms are known by Equations~\eqref{eq:5k+1}, \eqref{eq:5k} and \eqref{eq:5k+3}. 
We get the recurrence relation
$$\apu(5\m+2, 2\m+1) =  1+ \left(\frac{1}{2}\K^2 + \frac{5}{2}\K + 1 \right)  + (\K + 1)  + \apu(5(\K-1)+2, 2(\K - 1)+1),$$
where after $\m$ iterations we get the base case $\apu(2,1)$, the number of one-vertex LD*-codes in $P_2$. 
There are two such codes, namely $\{v_1\}$ and $\{v_2\}$.
Thus, we obtain
\begin{equation*}
\begin{split}
    \apu(5\K+2, 2\K+1)
    & =\frac{1}{2}\K^2 + \frac{7}{2}\K + 3 + \apu(5(\K-1)+2, 2(\K - 1)+1) \\
    &= \sum_{i=1}^{\K}\left(\frac{1}{2}i^2 + \frac{7}{2}i + 3 \right) + \apu(2,1) \\
    &=  \frac{1}{6}\K^3 + 2\K^2 + \frac{29}{6}\K + 2.
\end{split}
\end{equation*}

\noindent {\bf Equation~\eqref{eq:5k+4}}: Let us verify that $\apu(5\m + 4, 2\m + 2) = \frac{1}{24}\K^4 + \frac{11}{12}\K^3 + \frac{131}{24}\K^2 + \frac{115}{12}\K + 5$.
By~\eqref{eq:apukoodien-lkm},
\begin{align*}
    \apu(5\m + 4, 2\m + 2) &= \apu(5\m + 4 - 1, 2\m + 2 - 1) 
            + \apu(5\m + 4 - 2, 2\m + 2 - 1)\\
            &+ \apu(5\m + 4 - 4, 2\m + 2 - 2) 
            + \apu(5\m + 4 - 5, 2\m + 2 - 2), 
\end{align*}
where the first three terms are known by Equations~\eqref{eq:5k+3}, \eqref{eq:5k+2} and \eqref{eq:5k}. 
We get the recurrence relation
\begin{multline*}
    \apu(5\m+4, 2\m+2) =  (\K+2) +  \left(\frac{1}{6}\K^3 + 2\K^2 + \frac{29}{6}\K + 2\right) +\left(\frac{1}{2}\K^2 + \frac{5}{2}\K + 1\right) \\ + \apu(5(\K-1)+4, 2(\K - 1)+2)
\end{multline*}
where after $m$ iterations we get $\apu(4,2)$, the number of two-vertex LD*-codes in $P_4$. 
There are 5 such codes, namely all two-vertex subsets of $V(P_4)$ except $\{v_3, v_4\}$.
Thus, we obtain
\begin{equation*}
\begin{split}
        \apu(5\K + 4, 2\K + 2)
        &= \frac{1}{6}\K^3 + \frac{5}{2}\K^2 + \frac{25}{3}\K + 5 + \apu(5(\K-1)+4, 2(\K - 1)+2)\\
        &= \sum_{i=1}^{\K}\left(\frac{1}{6}i^3 + \frac{5}{2}i^2 + \frac{25}{3}i + 5\right) + \apu(4,2) \\
        &=  \frac{1}{24}\K^4 + \frac{11}{12}\K^3 + \frac{131}{24}\K^2 + \frac{115}{12}\K + 5.
\end{split}
\end{equation*}

%
%
Now we  have solved all the necessary values of $\apu(n, \koko)$, so let us 
calculate $\koodit(n)$,  the number of different minimum locating-dominating codes in $P_{n}$ for all $n\ge 5$ . In what follows, Cases~$M_0$--$M_4$ of the claim are shown.

\noindent {\bf Case~$M_0$}: To warm up, we calculate $C(5m)$ with our new method (although we already know from previous section that $C(5m)=1$). Now
Equation~\eqref{eq:koodien-lkm} gives us
\begin{align*}
    \koodit(5\m) = \apu(5\m - 1, 2\m- 1)
            + \apu(5\m - 3, 2\m - 2) 
            + \apu(5\m - 4, 2\m - 2). 
\end{align*}
The first two terms are eliminated by the bound~\eqref{eq:LD*bound}. 
We apply Equation~\eqref{eq:5k+1} and obtain that $\koodit(5\m) = \apu(5\m - 4, 2\m - 2) = 1$. 


\noindent {\bf Case~$M_1$}: Next, we calculate $\koodit(5\m+1)$.
Equation~\eqref{eq:koodien-lkm} gives us
\begin{align*}
    \koodit(5\m+1) = \apu(5\m, 2\m)
            + \apu(5\m -2, 2\m-1) 
            + \apu(5\m -3, 2\m-1) . 
\end{align*}
Now we respectively apply Equations~\eqref{eq:5k}, \eqref{eq:5k+3} and \eqref{eq:5k+2}, and we obtain
\begin{align*}
    \koodit(5\m+1) &= \left(\frac{1}{2}\m^2 + \frac{5}{2}\m + 1\right)
            + (\m - 1 + 2)
            + \left(\frac{1}{6}(\K-1)^3 + 2(\K-1)^2 + \frac{29}{6}(\K-1) + 2\right) \\ 
            &= \frac{1}{6}\m^3 + 2\m^2 + \frac{29}{6}\m + 1.
\end{align*}


\noindent {\bf Case~$M_2$}: Let us next calculate $\koodit(5\m+2)$.
Equation~\eqref{eq:koodien-lkm} gives us
\begin{align*}
    \koodit(5\m+2) 
            = \apu(5\m +1, 2\m)
            + \apu(5\m - 1, 2\m - 1) 
            + \apu(5\m - 2, 2\m - 1) . 
\end{align*}
After eliminating the middle term by \eqref{eq:LD*bound} 
and applying Equations~\eqref{eq:5k+1} and \eqref{eq:5k+3} to the other two terms, we get
\begin{equation*}
    \koodit(5\m+2) = 1
                    +(\m - 1 + 2) = \m+2.
\end{equation*}


\noindent {\bf Case~$M_3$}: Let us calculate $\koodit(5\m + 3)$.
Equation~\eqref{eq:koodien-lkm} gives us
\begin{align*}
    \koodit(5\m+3) = \apu(5\m +2, 2\m+1)
            + \apu(5\m, 2\m)
            + \apu(5\m -1, 2\m). 
\end{align*}
By applying Equations~\eqref{eq:5k+2}, \eqref{eq:5k} and \eqref{eq:5k+4}, we obtain
\begin{align*}
    \koodit(5\m+3) &= 
            \frac{1}{24}\m^4 + \frac{11}{12}\m^3 + \frac{131}{24}\m^2 + \frac{103}{12}\m + 3.
\end{align*}


\noindent {\bf Case~$M_4$}: Finally, let us compute $\koodit(5\m+4)$.
Equation~\eqref{eq:koodien-lkm} gives us
\begin{align*}
    \koodit(5\m+4) = \apu(5\m + 3, 2\m + 1)
            + \apu(5\m + 1, 2\m)
            + \apu(5\m, 2\m). 
\end{align*}
We apply Equations~\eqref{eq:5k+3}, \eqref{eq:5k+1} and \eqref{eq:5k} to get 
\vspace{\abovedisplayskip}

\hfill\(\displaystyle 
    \koodit(5\m+4) = 
    \frac{1}{2}\m^2 + \frac{7}{2}\m + 4.
\)
\end{proof}

\section{Computational complexity of determining min-forced and min-void vertices}\label{sec:complexity}

In this section, we show that the problems ``given a graph $G$ and a vertex $u \in V(G)$, is $u$ min-forced?'' and ``given a graph $G$ and a vertex $u \in V(G)$, is $u$ min-void?'' are co-NP-hard problems. 
Similar questions have been studied in the context of resolving sets and basis forced vertices in \cite{Hakanen_2022}.

We use the following notation for the 3-satisfiability problem (3-SAT).
Let $X = \{x_1, \dots, x_n\}$ be the set of variables and 
$\literals = \{x_1, \dots, x_n, \overline{x}_1, \dots, \overline{x}_n\}$ the set of literals, where $\overline{x}$ is the negation of $x$.
Let $\instance = \{ \clause_1, \clause_2, \cdots, \clause_m\}$ be a 3-SAT instance over $X$, where each clause $\clause_j \subseteq \literals$ contains exactly three literals.
An assignment of truth values is a set $A \subseteq U$ where either $x_i \in A$ or $\overline{x}_i \in A$ for each $x_i \in X$.
The logical interpretation of these definitions is that $\instance$ corresponds to the formula $\clause_1 \wedge \clause_2 \wedge \cdots \wedge \clause_m$ and the clause $\clause_j = \{u_{j, 1},u_{j, 2},u_{j, 3}\} \subseteq \literals$ corresponds to the formula $u_{j, 1} \vee u_{j, 2} \vee u_{j, 3}$.
In a truth assignment $\assignment$, the variable $x_i$ is set as \true if $x_i \in \assignment$ and \false if $\overline{x}_i \in \assignment$.
The truth assignment $\assignment$ satisfies the instance $\instance$ if each clause $\clause_j$ contain at least one true literal, in other words, $\clause_j \cap \assignment \neq \emptyset$. 

The proof presented here is inspired by the proof of Lemma 3.1 in \cite{Charon2003}.

\begin{theorem}\label{theo:complexity}
Let $G$ be a graph and $u \in V(G)$.
\begin{enumerate}[(i)]
    \item Deciding whether $u$ is a min-forced vertex of $G$ is a co-NP-hard problem.
    \item Deciding whether $u$ is a min-void vertex of $G$ is a co-NP-hard problem.
\end{enumerate}
\end{theorem}
\begin{proof}

We reduce the 3-SAT problem polynomially to each of the problems studied.

Let $X = \{x_1, \dots, x_n\}$ be a set of variables and $\instance = \{ \clause_1, \clause_2, \cdots, \clause_m\}$ be a 3-SAT instance over $X$. 
For each variable $x_i \in X$, we construct a graph $G_{x_i}$ in the following way:
\[V(G_{x_i}) = \{x_i, \overline{x}_i, a_i, b_i\},\]
\[E(G_{x_i}) = \{x_i a_i,\, a_i \overline{x}_i,\, \overline{x}_i b_i,\, b_i x_i\}.\]

For each clause $\clause_j \in \instance$, we construct a graph $G_{\clause_j}$ in the following way:
\[V(G_{\clause_j}) = \{\alpha_j, \beta_j, \gamma_j\},\]
\[E(G_{\clause_j}) = \{\alpha_j \beta_j,\, \beta_j \gamma_j\}.\]

We combine the variable graphs and clause graphs into $G=(V,E)$ as follows:
\[V = \bigcup_{i=1}^n V(G_{x_i}) \cup \bigcup_{j=1}^m V(G_{\clause_j}) \cup \{w, v\},\]
\[E = \bigcup_{i=1}^n E(G_{x_i}) \cup \bigcup_{j=1}^m E(G_{\clause_j}) \cup \{\alpha_j u \sep u \in \clause_j, 1 \leq j \leq m\,\} \cup \{w\alpha_j \sep 1 \leq j \leq m\} \cup \{wv\}.\]
The graph $G$ is illustrated in Figure~\ref{fig:esim-G}.
We see that $\abs{V(G)} = 4n + 3m + 2$ and $\abs{E(G)} = 4n + 6m + 1$.

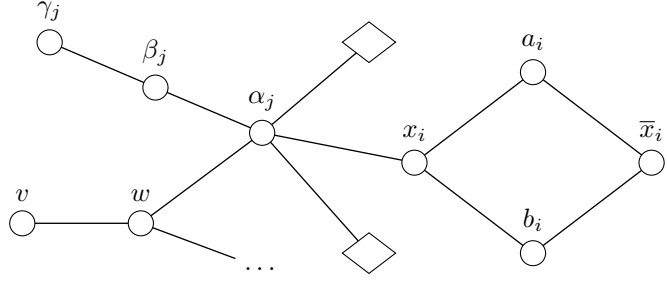
\begin{figure}
    \centering
    \begin{tikzpicture}[yscale=1.2, xscale=1.2, rotate=0]

    \renewcommand*{\EdgeLineWidth}{ 0.5pt}
    \SetVertexMath
    \SetVertexLabelOut

    {
    \tikzstyle{VertexStyle}=[circle,draw, fill=white]

    \Vertex[x=-.3, y=2, L=v, Lpos=90]{v}
    \Vertex[x=1, y=2, L=w, Lpos=90]{w}
    \Vertex[x=2.33, y=3, L=\alpha_j, Lpos=90]{a1}
    {
    \tikzstyle{VertexStyle}=[minimum size=1pt,inner sep=0pt, fill=none]
    \Vertex[x=3.5, y=4, L=\ , Lpos=0]{x33}
    \Vertex[x=3.5, y=1.67, L=\ , Lpos=0]{x11}
    \Vertex[x=3.5, y=4, L=\ , Lpos=0]{x3}
    \Vertex[x=3.5, y=1.67, L=\ , Lpos=0]{x1}
    }
    {
    \tikzstyle{VertexStyle}=[diamond, scale=1.3, aspect=1.3, draw, fill=white]
    \EA[unit=0, L=\ ](x11){tim1}
    \EA[unit=0, L=\ ](x33){tim3}
    }
    
    \Vertex[x=4, y=3-0.33, L=x_i, Lpos=90]{x2}

    \Vertex[x=5.3, y=4-0.33, L=a_i, Lpos=90]{a}
    \Vertex[x=6.6, y=3-0.33, L=\overline{x}_i, Lpos=90]{x5}
    \Vertex[x=5.3, y=2-0.33, L=b_i, Lpos=90]{b}

    \Vertex[x=1.16, y=3.5, L=\beta_j, Lpos=90]{bee}
    \Vertex[x=0, y=4, L=\gamma_j, Lpos=90]{gam}

    }
    {
    \tikzstyle{VertexStyle}=[fill=white,circle]
    \SetVertexLabelIn

    \Vertex[x=2.33, y=1.5, L=\dots]{dots}

    }

    \tikzstyle{EdgeStyle}=[bend right=0]
    \Edge(v)(w)    
    \Edge(w)(a1)
    \Edge(w)(dots)    

    \Edge(a1)(x2)
    \Edge(a1)(tim1)
    \Edge(a1)(tim3)
    
    \Edge(a1)(bee)
    \Edge(bee)(gam)
    
    \Edge(x2)(a)
    \Edge(x2)(b)
    \Edge(a)(x5)
    \Edge(b)(x5)

\end{tikzpicture}
    \caption{The vertex $w$ is connected to all vertices $\alpha_j$, represented by the ellipsis.
    The vertex $\alpha_j$ is connected to three variable gadgets, two of which are drawn as diamonds for the sake of simplicity.}
    \label{fig:esim-G}
\end{figure}

Regardless of whether $\instance$ is satisfiable, there is an LD-code $S$ in $G$ with $2n +m +1$ vertices:
\[S = \{ x_i, a_i \sep 1 \leq i \leq n\} \cup \{\beta_j \sep 1 \leq j \leq m\} \cup \{w\}.\]
The $I$-sets of vertices not belonging to $S$ are nonempty and distinct:
\begin{itemize}
    \setlength{\itemsep}{2pt}
    \setlength{\parskip}{2pt}
    \item $I_{G}(\alpha_j) \supseteq \{\beta_j, w\}$,
    \item $I_{G}(\gamma_j)=\{\beta_j\}$,
    \item $I_{G}(\overline{x}_i)=\{a_i\}$,
    \item $I_{G}(b_i)=\{x_i\}$, and
    \item $I_{G}(v)=\{w\}$.
\end{itemize}
Therefore, $\gamma^{LD}(G) \leq 2n +m +1$.

Next, we establish a lower bound for $\gamma^{LD}(G)$.
Each clause gadget $G_{\clause_{j}}$ necessarily contains at least one codeword to dominate $\gamma_j$: either $\beta_j$ or $\gamma_j$, since $N_G[\gamma_j] = \{\gamma_j, \beta_j\}$. 
Each variable gadget $G_{x_{i}}$ needs to contain at least two codewords: either $a_i$ or $b_i$, as they are twins, and one more to dominate the other one. 
Either $w$ or $v$ is in any LD-code in $G$, again,  since $N_G[v] = \{v, w\}$. Therefore, $\gamma^{LD}(G) = 2n +m +1$.

Let us consider some observations.

\begin{fact}
\label{fact:alpha}
    No minimum LD-code in $G$ contains the vertex $\alpha_j$ for any $j$.
That is because an LD-code containing $\alpha_j$ contains two codewords in the clause gadget $G_{\clause_{j}}$,
resulting in a code of at least $2n +m +2$ codewords.
\end{fact}

\begin{fact}
\label{fact:xx}
Exactly one of $x_i$ and $\overline{x}_i$ is in a minimum LD-code $S \subseteq V(G)$.
It follows from Fact~\ref{fact:alpha} that if $x_i, \overline{x}_i \notin S$, then $I_G(x_i), I_G(\overline{x}_i)  \subseteq \{a_i, b_i\}$ in any minimum LD-code $S \subseteq V(G)$.
Hence, $I_G(x_i) = I_G(\overline{x}_i)$ if $x_i, \overline{x}_i \notin S$.
Including both $x_i$ and $\overline{x}_i$ in $S$ would result in an LD-code of  at least $2n +m +2$ codewords.
Therefore, to distinguish the vertices $x_i$ and $\overline{x}_i$, exactly one of them is contained in every minimum LD-code. 
\end{fact}

\begin{fact}
\label{fact:wv}
    Exactly one of $w$ and $v$ is in a minimum LD-code $S \subseteq V(G)$.
\end{fact}

We claim that (i) the vertex $w$ is min-forced if and only if the 3-SAT instance $\instance$ is not satisfiable and (ii) the vertex $v$ is min-void if and only if the 3-SAT instance $\instance$ is not satisfiable.
We prove these statements by contraposition.

First, we assume that $\instance$ is satisfiable, and we show that $w$ is not min-forced and $v$ is not min-void.
If $\instance$ is satisfiable, then there exists a truth assignment $\assignment \subseteq \literals$ such that $\abs{\clause_j \cap \assignment } \geq 1$ for all clauses $\clause_j \in \instance$. 
Now
\[S = \assignment \cup \{a_i\sep 1 \leq i \leq n\} \cup \{\beta_j \sep 1 \leq j \leq m\} \cup \{v\}\]
is a minimum LD-code of $G$ with $v \in S$ and $w \notin S$.
The $I$-sets of vertices not belonging to $S$ are nonempty and distinct:
\begin{itemize}
\setlength{\itemsep}{2pt}
    \setlength{\parskip}{2pt}
    \item $I_{G}(\alpha_j) \supseteq \{\beta_j, u\}$, where $u$ is (one of) the true literal(s) in $\clause_j$, 
    \item $I_{G}(\gamma_j)=\{\beta_j\}$,
    \item if $x_i \in S$, then $I_{G}(\overline{x}_i)=\{a_i\}$, otherwise $\overline{x}_i \in S$ and $I_{G}(x_i)=\{a_i\}$,
    \item if $x_i \in S$, then $I_{G}(b_i)=\{x_i\}$, otherwise $\overline{x}_i \in S$ and  $I_{G}(b_i)=\{\overline{x}_i\}$, and
    \item $I_{G}(w)=\{v\}$.
\end{itemize}
In this case, we see that $w$ is not min-forced and $v$ is not min-void.

Next, we separately assume that $w$ is not min-forced and that $v$ is not min-void, and in both cases it follows that $F$ is satisfiable.
Here the proofs of claims (i) and (ii) briefly diverge.
\begin{enumerate}[(i)]
    \item Assume that $w$ is not min-forced, that is, there exists a minimum LD-code $S \subseteq V(G)$ such that $w \notin S$.
Recall Fact~\ref{fact:alpha}: the vertices $\alpha_j$ are never in a minimum LD-code of $G$.
Consider the $I$-set of the vertex $\alpha_j$ corresponding to the clause $\clause_j = \{u_{j,1},u_{j,2},u_{j,3}\}$.
If $\beta_j \in I(\alpha_j)$, then $\gamma_j$ is not a codeword and $I(\gamma_j) = \{\beta_j \}$.
For $I(\alpha_j)$ to be distinct, there must be $u_{j,k} \in I(\alpha_j)$ for some $k \in \{1,2,3\}$.
If $\beta_j \notin I(\alpha_j)$, then for $I(\alpha_j)$ to be nonempty, there must be $u_{j,k} \in I(\alpha_j)$ for some $k \in \{1,2,3\}$.

By Fact~\ref{fact:xx}, either $x_i \in S$ or $\overline{x}_i \in S$, but not both.
When we set the variable $x_i$ as \true if $x_i \in S$ and as \false if $\overline{x}_i \in S$, we get a valid truth assignment for the variables of $X$.
Under this assignment, each clause $\clause_j$ in $\instance$ contains at least one true literal, namely, the literal $u_{j,k}$ that is in the $I$-set of $\alpha_j$.
Therefore, $\instance$ is satisfiable.

    \item Now we assume that $v$ is not min-void, that is, there exists a minimum LD-code $S \subseteq V(G)$ such that $v \in S$.
    By Fact~\ref{fact:wv}, $v \in S$ implies that $w \notin S$, in other words, $w$ is not min-forced.

    Just like in the proof of claim (i), it follows that $\instance$ is satisfiable. 
\end{enumerate}

We have shown that the vertex $w$ is min-forced (respectively, $v$ is min-void) if and only if the 3-SAT instance $\instance$ is not satisfiable.
Thus, there exists a polynomial-time reduction of the complement of the 3-SAT problem to the problem of deciding whether a given vertex is min-forced (resp. min-void).
Hence, the studied problems are co-NP-hard.
\end{proof}

Note that the graph $\graphminus{G}{\{v, w\}}$ can be used to show that the decision problem ``is there a locating-dominating code $S\subseteq V(G)$ of size at most $k$?'' is NP-complete (see Lemma 3.1 in \cite{Charon2003}).
The graph $\graphminus{G}{\{v, w\}}$ has fewer vertices than the graph used in the proof in \cite{Charon2003}, which in turn needs fewer vertices compared to the proof of Colbourn et al. \cite{colbourn1987locating}.

\section*{Acknowledgements}

The authors have been partially supported by Research Council of Finland grant number 338797. Havu Miikonen has been partially supported by the Turku University Foundation. The authors would like to thank the anonymous referee for careful reading of the paper and insightful comments.

\bibliographystyle{abbrvnat}
\bibliography{refs}

@article {Blidia2009minforcedtrees,
    AUTHOR = {Blidia, Mostafa and Lounes, Rahma},
     TITLE = {Vertices belonging to all or to no minimum locating dominating
              sets of trees},
    JOURNAL = {Opuscula Math.},
    VOLUME = {29},
      YEAR = {2009},
    NUMBER = {1},
     PAGES = {5--14},
      ISSN = {1232-9274,2300-6919},
   MRCLASS = {05C69},
  MRNUMBER = {2480428},
       DOI = {10.7494/OpMath.2009.29.1.5},
       URL = {https://doi.org/10.7494/OpMath.2009.29.1.5},
}

@article {hernando2018locating,
    AUTHOR = {Hernando, C. and Mora, M. and Pelayo, I. M.},
     TITLE = {Locating domination in bipartite graphs and their complements},
   JOURNAL = {Discrete Appl. Math.},
  FJOURNAL = {Discrete Applied Mathematics. The Journal of Combinatorial
              Algorithms, Informatics and Computational Sciences},
    VOLUME = {263},
      YEAR = {2019},
     PAGES = {195--203},
      ISSN = {0166-218X,1872-6771},
   MRCLASS = {05C69},
  MRNUMBER = {3956048},
       DOI = {10.1016/j.dam.2018.09.034},
       URL = {https://doi.org/10.1016/j.dam.2018.09.034},
}

@article{Hakanen_2022,
	doi = {10.1016/j.dam.2021.12.004},
  
	url = {https://doi.org/10.1016/j.dam.2021.12.004},
  
	year = 2022,
	month = {oct},
  
	publisher = {Elsevier {BV}},
  
	volume = {319},
  
	pages = {407--423},
  
	author = {Anni Hakanen and Ville Junnila and Tero Laihonen and Ismael G. Yero},
  
	title = {On vertices contained in all or in no metric basis},
  
	journal = {Discrete Appl. Math.}
}

@article{BOUQUET20219,
title = {On the vertices belonging to all, some, none minimum dominating set},
journal = {Discrete Appl. Math.},
volume = {288},
pages = {9-19},
year = {2021},
issn = {0166-218X},
doi = {https://doi.org/10.1016/j.dam.2020.08.020},
url = {https://www.sciencedirect.com/science/article/pii/S0166218X20303863},
author = {Valentin Bouquet and François Delbot and Christophe Picouleau}
}

@article{Charon2003,
  doi = {10.1016/s0304-3975(02)00536-4},
  url = {https://doi.org/10.1016/s0304-3975(02)00536-4},
  year = {2003},
  month = jan,
  publisher = {Elsevier {BV}},
  volume = {290},
  number = {3},
  pages = {2109--2120},
  author = {Ir{\`{e}}ne Charon and Olivier Hudry and Antoine Lobstein},
  title = {Minimizing the size of an identifying or locating-dominating code in a graph is {NP}-hard},
  journal = {Theoret. Comput. Sci.}
}

@article{Hudry2019,
  title = {Unique (optimal) solutions: Complexity results for identifying and locating–dominating codes},
  volume = {767},
  ISSN = {0304-3975},
  url = {http://dx.doi.org/10.1016/j.tcs.2018.09.034},
  DOI = {10.1016/j.tcs.2018.09.034},
  journal = {Theoret. Comput. Sci.},
  publisher = {Elsevier BV},
  author = {Hudry,  Olivier and Lobstein,  Antoine},
  year = {2019},
  month = may,
  pages = {83--102}
}

@article{colbourn1987locating,
  title={Locating dominating sets in series parallel networks},
  author={Colbourn, Charles J and Slater, Peter J and Stewart, Lorna K},
  journal={Congr. Numer.},
  volume={56},
  number={1987},
  pages={135--162},
  year={1987}
}

@article{slater1987domination,
author = {Slater, Peter J.},
title = {Domination and location in acyclic graphs},
journal = {Networks},
volume = {17},
number = {1},
pages = {55--64},
doi = {https://doi.org/10.1002/net.3230170105},
url = {https://onlinelibrary.wiley.com/doi/abs/10.1002/net.3230170105},
eprint = {https://onlinelibrary.wiley.com/doi/pdf/10.1002/net.3230170105},
year = {1987}
}

@article{Boros,
    AUTHOR = {Boros, Endre and Golumbic, Martin C. and Levit, Vadim E.},
     TITLE = {On the number of vertices belonging to all maximum stable sets
              of a graph},
      NOTE = {Workshop on Discrete Optimization (Piscataway, NJ, 1999)},
   JOURNAL = {Discrete Appl. Math.},
  FJOURNAL = {Discrete Applied Mathematics. The Journal of Combinatorial
              Algorithms, Informatics and Computational Sciences},
    VOLUME = {124},
      YEAR = {2002},
    NUMBER = {1-3},
     PAGES = {17--25},
      ISSN = {0166-218X,1872-6771},
   MRCLASS = {05C35 (05C85)},
  MRNUMBER = {1924949},
       DOI = {10.1016/S0166-218X(01)00327-4},
       URL = {https://doi.org/10.1016/S0166-218X(01)00327-4},
}

@online{jeanlobwww,
  title        = {Watching systems, identifying, locating-dominating and discriminating codes in graphs},
  author       = "Jean, Devin and Lobstein, Antoine",
  year         = 2024,
  url          = {https://dragazo.github.io/bibdom/main.pdf},
  urldate      = {5-June-2024},
note = {Accessed: June 5th 2024},
howpublished = {Published electronically}
}

@article{rall1984location,
  title={On location-domination numbers for certain classes of graphs},
  author={Rall, Douglas F and Slater, Peter J},
  journal={Congr. Numer.},
  volume={45},
  pages={97--106},
  year={1984}
}

@article{slater1988dominating,
  title={Dominating and reference sets in a graph},
  author={Slater, Peter J},
  journal={J. Math. Phys. Sci.},
  volume={22},
  number={4},
  pages={445--455},
  year={1988}
}

\end{document}